\documentclass[oneside,11pt,a4paper,reqno,ps]{article}
\usepackage[a4paper,top=2.5cm,bottom=2.5cm,left=2.5cm,right=2.5cm,
bindingoffset=5mm]{geometry}
\usepackage[latin1]{inputenc}
\usepackage[english]{babel}
\usepackage[T1]{fontenc}
\usepackage{syntonly}
\usepackage{graphicx}
\usepackage{amssymb}	
\usepackage{amsthm}
\usepackage{lmodern}
\usepackage{microtype}
\usepackage{amsmath}
\usepackage{braket}
\usepackage{eurosym}
\usepackage{amssymb}
\usepackage{amstext}
\usepackage{color}
\usepackage{xcolor}
\usepackage{url}
\usepackage{amsmath,amsthm,amssymb,amscd,epsfig,esint, mathrsfs,graphics,color}
\usepackage{wrapfig}
\usepackage{verbatim}
\usepackage{hyperref}
\usepackage{listings}
\usepackage{bookmark}
\usepackage{csquotes}
\hypersetup{%
	pdfpagemode={UseOutlines},
	bookmarksopen,
	pdfstartview={FitH},
	hidelinks
}
\usepackage[backend=bibtex, hyperref]{biblatex}
\newcommand{\R}{\numberset{R}}

\newcommand{\lnorm}{\left\Arrowvert}
\newcommand{\rnorm}{\right\Arrowvert}

\theoremstyle{plain}
\newtheorem{thm}{Theorem}[section]

\newtheorem{proposition}[thm]{Proposition}
\newtheorem{lemma}[thm]{Lemma}
\newtheorem{corollary}[thm]{Corollary}

\theoremstyle{definition}
\newtheorem{definition}[thm]{Definition}

\newtheorem{remark}[thm]{Remark}

\def\Xint#1{\mathchoice
	{\XXint\displaystyle\textstyle{#1}}%
	{\XXint\textstyle\scriptstyle{#1}}%
	{\XXint\scriptstyle\scriptscriptstyle{#1}}%
	{\XXint\scriptscriptstyle\scriptscriptstyle{#1}}%
	\!\int}
\def\XXint#1#2#3{{\setbox0=\hbox{$#1{#2#3}{\int}$}
		\vcenter{\hbox{$#2#3$}}\kern-.5\wd0}}
\def\Mint{\Xint -}

\def\eps{\varepsilon}

\def\eps{\varepsilon}
\def\div{{\rm div}}

\def\R{\mathbb{R}}
\def\e{\varepsilon}

\def\supp{\mathrm{supp}}
\def\loc{\mathrm{loc}}
\numberwithin{equation}{section} \makeatletter

\renewcommand{\p@enumi}{\thesection.}
\title{\textbf{Regularity results for solutions to a class of non-autonomous obstacle problems with sub-quadratic growth conditions}}
\author{Andrea Gentile, Raffaella Giova}
\begin{document}
	
	\maketitle

	\begin{abstract}
		\noindent We establish some higher differentiability results for solution to non-autonomous obstacle problems of the form
		
		\begin{equation*}
			\min \left\{\int_{\Omega}f\left(x, Dv(x)\right)dx\,:\, v\in
			\mathcal{K}_\psi(\Omega)\right\},
		\end{equation*}
		
		\noindent where the function $f$ satisfies $p-$growth conditions with respect to the gradient variable, for $1<p<2$, and $\mathcal{K}_\psi(\Omega)$ is the class of admissible functions.\\
		Here we show that, if the obstacle $\psi$ is bounded, then a Sobolev regularity assumption on the gradient of the obstacle $\psi$ transfers to the gradient of the solution, provided the partial map $x\mapsto D_\xi f(x,\xi)$ belongs to a Sobolev space, $W^{1, p+2}$.\\
		The novelty here is that we deal with subquadratic growth conditions with respect to the gradient variable, i.e. $f(x, \xi)\approx a(x)|\xi|^p$ with $1<p<2,$ and where the map $a$ belongs to a Sobolev space.
	\end{abstract}
	
	\bigskip
	\section{Introduction}
	
	We are interested in the regularity properties of solutions to problems of the form
	
	\begin{equation}\label{functionalobstacle}
		\min \Set{\int_{\Omega}f\left(x, Dv(x)\right)dx\,:\, v\in
		\mathcal{K}_\psi(\Omega)},
	\end{equation}
	where $\Omega\subset\R^n$ is a bounded open set, $n>2$, $f:\Omega\times\R^n\to\R$ is a Carath\'{e}odory map, such that $\xi\mapsto f(x, \xi)$ is of class $C^2(\R^n)$ for a. e. $x\in\Omega$, $\psi: \Omega
	\mapsto [-\infty, +\infty)$ belonging to the Sobolev class $ W^{1,p}_{\loc}\left(\Omega\right)$ is the \emph{obstacle}, and
	
	$$\mathcal{K}_\psi(\Omega)=\Set{v\in u_0+W^{1, p}_0(\Omega, \R): v\ge\psi \text{ a.e. in }\Omega}$$
	
	is the class of the admissible functions, with $u_0\in W^{1, p}(\Omega)$ a fixed boundary datum.
	
	\bigskip

	Let us observe that $u\in W^{1,p}_{\loc}\left(\Omega\right)$ is a
	solution to the obstacle problem \eqref{functionalobstacle} in
	$\mathcal{K}_\psi(\Omega)$ if and only if $u \in
	\mathcal{K}_\psi(\Omega)$ and $u$ is a solution to the variational
	inequality
	
	\begin{equation}\label{variationalinequality}
		\int_{\Omega}\left<A(x, Du(x)), D(\varphi(x)-u(x))\right>dx\ge0\qquad\forall
		\varphi\in \mathcal{K}_\psi(\Omega),
	\end{equation}
	where the operator $A: \Omega\times\R^n\to\R^n$ is defined
	as follows
	
	\begin{equation*}
		A_i(x, \xi)=D_{\xi_i}f(x, \xi)\qquad\forall i=1,...,n.
	\end{equation*}
	
	We assume that $A$ is a $p$-harmonic type operator, that is it satisfies the following $p$-ellipticity and $p$-growth conditions with respect to
	the $\xi$-variable. There exist positive constants $\nu, L, \ell$
	and an exponent $1<p\le2$ and a parameter $0\le \mu\le 1$ such that
	
	\begin{equation}\label{obstacleA3}
		\left|A(x, \xi)\right|\le\ell\left(\mu^2+|\xi|^2\right)^\frac{p-1}{2},
	\end{equation}
	
	\begin{equation}\label{obstacleA1}
		\left<A(x, \xi)-A(x, \eta), \xi-\eta\right>\ge\nu|\xi-\eta|^2\left(\mu^2+|\xi|^2+|\eta|^2\right)^\frac{p-2}{2},
	\end{equation}
		
	\begin{equation}\label{obstacleA2}
		\left|A(x, \xi)-A(x, \eta)\right|\le L|\xi-\eta|\left(\mu^2+|\xi|^2+|\eta|^2\right)^\frac{p-2}{2},
	\end{equation}
	
	for all $\xi, \eta\in\R^n$ and for almost every $x\in\Omega.$\\
	
	First, we show that an higher differentiability property of integer order of the gradient of the obstacle tranfers to the solution of problem \eqref{functionalobstacle}, provided the partial map $x\mapsto D_\xi f(x, \xi)$ belongs to a suitable Sobolev class, with no loss in the order of differentiation.\\
	More precisely we assume that the map $x\mapsto A\left(x,\xi\right)$ belongs to $W^{1,p+2}_{\loc}\left(\Omega	\right)$ for every $\xi\in\R^n$ or, equivalently, that there exists a non-negative function $g\in L^{p+2}_{\loc}\left(\Omega\right)$ such that
	
	\begin{equation}\label{x-dependenceH}
			\left|A\left(x, \xi\right)-A\left(y, \xi\right)\right|\le
			\left(g(x)+g(y)\right)\left|x-y\right|\left(\mu^2+\left|\xi\right|^2\right)^\frac{p-1}{2}
		\end{equation}
		for a. e. $x, y\in\Omega$ and for every $\xi\in\R^n$, and
	
	\begin{equation}\label{x-dependence}
		\left|D_xA\left(x, \xi\right)\right|\le
		g(x)\left(\mu^2+\left|\xi\right|^2\right)^\frac{p-1}{2},
	\end{equation}
	(see \cite{H}).\\
	
	Note that, since $f$, as a function of the $\xi$ variable, is of class $C^2$, then the operator $A$ is of class $C^1$ with respect to $\xi$, and \eqref{obstacleA2} implies
	
	\begin{equation}\label{A2bis}
		\left|D_\xi A(x, \xi)\right|\le c\left(\mu^2+|\xi|^2\right)^\frac{p-2}{2},
	\end{equation}
	
	for all $\xi\in\R^n\setminus\Set{0}$ and for a.e. $x\in\Omega.$\\
	
Higher differentiability results for solutions to homogeneous obstacle problems of the type \eqref{functionalobstacle} under a Sobolev assumption on the partial map $x\mapsto A(x, \xi)$ have been obtained when the energy density satisfies standard $p$-growth conditions, both in case $p\ge 2$ (\cite{EP}) and in case $1<p<2$ (\cite{Gentile3}). The case of non-standard growth conditions has been faced, for example, in \cite{Gavioli2, Gavioli}, for what concerns $(p,q)$-growth and in case of variable exponents in \cite{FG}. The case of non-homogeneous obstacle problems is faced in \cite{MaZ}, where the energy density satisfies $p$-growth conditions, and in \cite{CEP}, where the energy density satisfies $(p,q)$-growth conditions. All previously quoted higher differentiability results have been obtained assuming that, with respect to the $x$-variable, the map $A$ belongs to a Sobolev space $W^{1,r}$ with $r\ge n$.\\ 
However, taking into account the result obtained for unconstrained problem in \cite{CGGP1, GiovaPassarelli}, proving that if we deal with
bounded solutions to, the higher differentiability holds true under weaker assumptions on the partial map $x\mapsto A(x, \xi)$ with respect to $ W^{1,n}$, and the result obtained in \cite{CEP} proving that a
local bound assumption on the obstacle $\psi$ implies a local bound for the solutions to the obstacle problem \eqref{functionalobstacle}, in \cite{CGG, GGT} have been proven that, if the obstacle is locally bounded, higher differentiability results for the solutions of \eqref{functionalobstacle} persist assuming that the partial map $x\mapsto A(x, \xi)$ belongs to a Sobolev class that is not related to the dimension $n$ but to the growth exponent $p$ of the functional in case of standard growth and to the ellipticity and the growth exponents $p$ and $q$ of the functional in case of non-standard growth.\\ 
We want to observe that all these regularity results obtained to the solutions of obstacle problems have been inspired by the results obtained for the solutions of systems or for the minimizers of functionals in the case of unconstrained problems (see \cite{ CUMR, EleMarMas2, EleMarMas, EleMarMas3, Gentile2, Gentile1, Giova1, Giova2, APdN1, APdN2} ) since the regularity of the solutions to the obstacle problem
\eqref{functionalobstacle} is strictly connected to the analysis of the regularity of the solutions to a partial differential equation of the form
$$
\div A(x,Du) = \div A(x,D\psi),
$$
(see Theorem \ref{Fuchs} in Section \ref{Preliminaries} below).\\
For higher fractional differentiability results, we refer to \cite {EP, GRIMALDI2021103377, GrimaldiIpocoana, MaZ, ZhangZheng}  for obstacle problems, and to \cite{AMBROSIO2022125636, BCO, BCGOP, BDW, CGP, Giova3} for the case of unconstrained problems.\\
For other results dealing regularity of solutions to elliptic problems, both in case of unconstrained and constrained case, we refer to \cite{ Beck-Mingione, de2021lipschitz, de2020regularity}, where Lipschitz regularity results are proved, even under non-standard growth and ellipticity conditions.

As usual in the study of regularity results for solutions to elliptic problems, we shall use a function of the gradient defined as
	$$
	V_p\left(\xi\right)=\left(\mu^2+
	\left|\xi\right|^2\right)^\frac{p-2}{4}\xi, \qquad \mbox{ for any }\xi\in\R^n.
	$$
The main result we prove in this paper is the following.
	
\begin{thm}\label{thm1GG}
	Let $u\in W^{1,p}_{\loc}\left(\Omega\right)$ be a solution to the obstacle problem \eqref{functionalobstacle} under assumptions \eqref{obstacleA3}--\eqref{obstacleA2} and let us assume that there exists a function $g\in L^{p+2}_\loc\left(\Omega\right)$ such that \eqref{x-dependenceH} and \eqref{x-dependence} hold, for $1<p<2$.\\
	Then the following implication holds:
	
	\begin{equation*}\label{implication1GG}
		\psi\in L^\infty_{\loc}\left(\Omega\right) \mbox{ and } V_p\left(D\psi\right)\in W^{1,2}_{\loc}\left(\Omega\right)\implies V_p\left(Du\right)\in W^{1,2}_{\loc}\left(\Omega\right).
	\end{equation*}
	
	Moreover, for any ball $B_{8R}\Subset\Omega$, the following estimate holds
	
	\begin{eqnarray}\label{estimate1GG}
		&&\int_{B_{\frac{R}{2}}}\left|DV_p\left(Du(x)\right)\right|^2dx\cr\cr
		&\le&\frac{c\left(\left\Arrowvert \psi\right\Arrowvert_{L^\infty\left(B_{8R}\right)}^2+\left\Arrowvert u\right\Arrowvert_{L^{p^*}\left(B_{8R}\right)}^2\right)^{\sigma_1}}{R^2}\cr\cr
		&&\cdot\left[\int_{B_{4R}}\left(\mu^2+\left|Du(x)\right|^2\right)^\frac{p}{2}dx+\int_{B_{4R}}g^{p+2}(x)dx\right.\cr\cr
		&&\left.+\int_{B_{4R}} \left|DV_p\left(D\psi(x)\right)\right|^2dx+\int_{B_{4R}}\left(\mu^2+\left|D\psi(x)\right|^2\right)^\frac{p}{2}dx\right]^{\sigma_2},
	\end{eqnarray}
	
	where $c>0$ depends on $n, p, \nu, L$ and $\ell$ and $\sigma_1, \sigma_2>0$ depend on $n$ and $p$.
\end{thm}
	
	Let us oberve that, since we are taking $1<p<2$, if we have $p+2<n$ (i.e. if $n\ge4$), under stronger assumptions on the regularity of the obstacle, we can assume less regularity on the coefficients, if we compare this result with the Theorem 1.1 in \cite{Gentile3}, where $g\in L^n$, but the obstacle is not assumed to be bounded. However, the same result can be obtained by removing the assumption of boundedness of the obstacle if we consider a priori bounded minimizers. \\
	Moreover, we can see this result as an extention to the case of sub-quadratic growth and ellipticity, of the result proved in \cite{CGG}.
	
	\section{Notations and preliminary results}\label{Preliminaries}
	
	In this section we list the notations that we use in this paper and recall some tools that will be useful to prove our results.\\
	We shall follow the usual convention and denote by $C$ or $c$ a general constant that may vary on different occasions, even within the same
	line of estimates. Relevant dependencies on parameters and special
	constants will be suitably emphasized using parentheses or
	subscripts. The norm we use on $\R^n$, will be the standard Euclidean one.\\
	For a $C^2$ function $f \colon \Omega\times\R^{n} \to \R$, we write
	$$
	D_\xi f(x,\xi )[\eta ] := \frac{\rm d}{{\rm d}t}\Big|_{t=0} f(x,\xi
	+t\eta )\quad \mbox{ and } \quad D_{\xi\xi}f(x,\xi )[\eta ,\eta ] :=
	\frac{\rm d^2}{{\rm d}t^{2}}\Big|_{t=0} f(x,\xi +t\eta )
	$$
	for $\xi$, $\eta \in \R^{n}$ and for almost every $x\in \Omega$.\\
	With the symbol $B(x,r)=B_r(x)=\{y\in
	\R^n:\,\, |y-x|<r\}$, we will denote the ball centered at $x$ of
	radius $r$ and
	$$(u)_{x_0,r}= \Mint_{B_r(x_0)}u(x)\,dx,$$
	stands for the integral mean of $u$ over the ball $B_r(x_0)$. We
	shall omit the dependence on the center when it is clear from the context.
	In the following, we will denote, for any ball
	$B=B_r(x_0)=\{x\in\R^n: |x-x_0|<r\}\Subset\Omega$
	
	\begin{equation*}
		\Mint_Bu(x)dx=\frac{1}{|B|}\int_Bu(x)dx.
	\end{equation*}
	
Here we recall some results that will be useful in the following.\\
The next lemma can be proved using an iteration technique, so we will refer to this as \text{Iteration Lemma}.

\begin{lemma}[Iteration Lemma]\label{iteration}
	Let $h: [\rho, R]\to \R$ be a nonnegative bounded function, $0<\theta<1$, $A, B\ge0$ and $\gamma>0$. Assume that
	
	$$
	h(r)\le\theta h(d)+\frac{A}{(d-r)^\gamma}+B
	$$
	
	for all $\rho\le r<d\le R_0.$ Then
	
	$$
	h(\rho)\le c\left[\frac{A}{(R_0-\rho)^\gamma}+B\right],
	$$
	
	where $c=c(\theta, \gamma)>0$.
\end{lemma}
For the proof we refer to \cite[Lemma 6.1]{23}.\\
\begin{lemma}\label{lemma5GP}
	For any $\phi\in C_0^1(\Omega)$ with $\phi\ge0$, and any $C^2$ map $v:\Omega\to\R^N$, we have
	
	\begin{eqnarray}\label{2.1GP}
		&&\int_\Omega\phi^{\frac{m}{m+1}(p+2)}(x)\left|Dv(x)\right|^{\frac{m}{m+1}(p+2)}dx\cr\cr
		&\le&(p+2)^2\left(\int_\Omega\phi^{\frac{m}{m+1}(p+2)}(x)\left|v(x)\right|^{2m}dx\right)^\frac{1}{m+1}\cr\cr
		&&\cdot\left[\left(\int_\Omega\phi^{\frac{m}{m+1}(p+2)}(x)\left|D\phi(x)\right|^2\left|Dv(x)\right|^pdx\right)^\frac{m}{m+1}\right.\cr\cr
		&&\left.+n\left(\int_\Omega\phi^{\frac{m}{m+1}(p+2)}(x)\left|Dv(x)\right|^{p-2}\left|D^2v(x)\right|^2dx\right)^\frac{m}{m+1}\right],
	\end{eqnarray}
	
	for any $p\in(1, \infty)$ and $m>1$. Moreover, for any $\mu\in[0,1]$
	
	\begin{eqnarray}\label{2.2GP}
		&&\int_{\Omega}\phi^2(x)\left(\mu^2+\left|Dv(x)\right|^2\right)^\frac{p}{2}\left|Dv(x)\right|^2dx\cr\cr
		&\le&c\left\Arrowvert v\right\Arrowvert_{L^\infty\left(\supp(\phi)\right)}^2\int_\Omega\phi^2(x)\left(\mu^2+\left|Dv(x)\right|^2\right)^\frac{p-2}{2}\left|D^2v(x)\right|^2dx\cr\cr
		&&+c\left\Arrowvert v\right\Arrowvert_{L^\infty\left(\supp(\phi)\right)}^2\int_\Omega\left(\phi^2(x)+\left|D\phi(x)\right|^2\right)\left(\mu^2+\left|Dv(x)\right|^2\right)^\frac{p}{2}dx,
	\end{eqnarray}
	
	for a constant $c=c(p).$
\end{lemma}

The following result is proved for solutions to the problem \eqref{functionalobstacle} as Theorem 1.1 in \cite{CEP} in the case of $p,q-$growth conditions with respect to the gradient variable with $2\le p\le q$, but the proof still works exactly in the same way under sub-quadratic growth conditions, so that for $1<p=q<2$ this suits with our ellipticity and growth assumptions.

\begin{thm}\label{boundedness}
	Let u in $\mathcal{K}_{\psi}(\Omega)$ be a solution of
	\eqref{functionalobstacle} under the assumptions $\eqref{obstacleA3}-\eqref{obstacleA1}$. If
	the obstacle $\psi \in L^{\infty}_{\loc}\left(\Omega\right)$, then $u
	\in L^{\infty}_{\loc}\left(\Omega\right)$ and the following estimate
	\begin{equation}\label{est boundedness}
	\lnorm u\rnorm_{L^\infty\left(B_{\frac{R}{2}}\right)} \le \left[ \lnorm \psi\rnorm_{L^\infty(B_{R})}+ \left(
		\int_{B_R}\left|u(x)\right|^{p^*} dx\right)\right]^{\gamma}
	\end{equation}
	holds for every ball $B_R \Subset \Omega$, for $\gamma(n,p)>0$ and
	$c=c(\ell, \nu, p, n)$,
\end{thm}

	Let us conclude this section with a result (see \cite{Fuchs, FM}), that is useful to prove Theorem \ref{thm1GG}.
	
	\begin{thm}\label{Fuchs}
		For any $1<p<\infty$, a function $u\in W^{1, p}_{\loc}\left(\Omega\right)$ is a solution to the problem \eqref{functionalobstacle} if and only if it is a weak solution of the following equation:
		
		\begin{equation}\label{diveq}
			\div A\left(x, Du(x)\right)=-\div A(x, D\psi(x))\chi_{\set{u=\psi}}(x).
		\end{equation}
	\end{thm}
	
\subsection{Difference quotients}\label{diffquot}
A key instrument in studying regularity properties of solutions to problems of Calculus of Variations and PDEs is the so called {\em difference quotients method}.\\
In this section, we recall the definition and some basic results.
\begin{definition}
	Given $h\in\R$, for every function
	$F:\mathbb{R}^{n}\to\mathbb{R}^N$, for any $s=1,..., n$ the finite difference operator in the direction $x_s$ is
	defined by
	$$
	\tau_{s, h}F(x)=F(x+he_s)-F(x),
	$$
	where $e_s$ is the unit vector in the direction $x_s$.
\end{definition}
In the following, in order to simplify the notations, we will omit the vector $e_s$ unless it is necessary, denoting
$$
\tau_{h}F(x)=F(x+h)-F(x),
$$
where $h\in\R^n$.\\

We now describe some properties of the operator $\tau_{h}$ whose proofs can be found, for example, in \cite{23}.

\bigskip

\begin{proposition}\label{findiffpr}
	
	Let $F$ and $G$ be two functions such that $F, G\in
	W^{1,p}(\Omega)$, with $p\geq 1$, and let us consider the set
	$$
	\Omega_{|h|}:=\Set{x\in \Omega : d\left(x,
		\partial\Omega\right)>\left|h\right|}.
	$$
	Then
	\begin{description}
		\item{$(a)$} $\tau_{h}F\in W^{1,p}\left(\Omega_{|h|}\right)$ and
		$$
		D_{i} (\tau_{h}F)=\tau_{h}(D_{i}F).
		$$
		\item{$(b)$} If at least one of the functions $F$ or $G$ has support contained
		in $\Omega_{|h|}$ then
		$$
		\int_{\Omega} F(x) \tau_{h} G(x) dx =\int_{\Omega} G(x) \tau_{-h}F(x)
		dx.
		$$
		\item{$(c)$} We have
		$$
		\tau_{h}(F G)(x)=F(x+h )\tau_{h}G(x)+G(x)\tau_{h}F(x).
		$$
	\end{description}
\end{proposition}

\noindent The next result about finite difference operator is a kind
of integral version of Lagrange Theorem.
\begin{lemma}\label{le1} If $0<\rho<R$, $|h|<\frac{R-\rho}{2}$, $1<p<+\infty$,
	and $F, DF\in L^{p}(B_{R})$ then
	$$
	\int_{B_{\rho}} |\tau_{h} F(x)|^{p}\ dx\leq c(n,p)|h|^{p}
	\int_{B_{R}} |D F(x)|^{p}\ dx .
	$$
	Moreover
	$$
	\int_{B_{\rho}} |F(x+h )|^{p}\ dx\leq \int_{B_{R}} |F(x)|^{p}\ dx .
	$$
\end{lemma}

The following result is proved in \cite{23}.

\begin{lemma}\label{Giusti8.2}
	Let $F:\R^n\to\R^N$, $F\in L^p\left(B_R\right)$ with $1<p<+\infty$. Suppose that there exist $\rho\in(0, R)$ and $M>0$ such that
	
	$$
	\sum_{s=1}^{n}\int_{B_\rho}|\tau_{s, h}F(x)|^pdx\le M^p|h|^p
	$$
	
	for $\left|h\right|<\frac{R-\rho}{2}$. Then $F\in W^{1,p}(B_R, \R^N)$. Moreover
	
	$$
	\left\Arrowvert DF \right\Arrowvert_{L^p(B_\rho)}\le M,
	$$
	
	$$
	\left\Arrowvert F\right\Arrowvert_{L^{\frac{np}{n-p}}(B\rho)}\le c\left(M+\left\Arrowvert F\right\Arrowvert_{L^p(B_R)}\right),
	$$
	
	with $c=c(n, N, p, \rho, R)$. Moreover		
	$$\frac{\tau_{s, h}F}{\left|h\right|}\to D_sF\qquad\mbox{ in }L^p_{\loc}\left(\Omega\right),\mbox{ as }h\to0,$$
	for each $s=1, ..., n.$
\end{lemma}

\subsection{An auxiliary function}\label{VpSection}

Here we define an auxiliary function of the gradient variable that will be useful in the following.\\
The function $V_p:\R^{n}\to\R^{n}$, defined as

\begin{equation*}\label{Vp}
	V_p(\xi):=\left(\mu^2+\left|\xi\right|^2\right)^\frac{p-2}{4}\xi,
\end{equation*}

\noindent for which the following estimates hold (see \cite{AF}).

\begin{lemma}\label{lemma6GP}
	Let $1<p<2$. There is a constant $c=c(n, p)>0$ such that
	
	\begin{equation}\label{lemma6GPestimate1}
		c^{-1}\left|\xi-\eta\right|\le\left|V
		_p(\xi)-V_p(\eta)\right|\cdot\left(\mu^2+\left|\xi\right|^2+\left|\eta\right|^2\right)^\frac{2-p}{4}\le c\left|\xi-\eta\right|,
	\end{equation}
	\noindent for any $\xi, \eta\in\R^n.$
\end{lemma}

\begin{remark}\label{rmk1}
	One can easily check that, for a $C^2$ function $v$, there is a constant $C(p)$ such that
		\begin{equation}\label{lemma6GPestimate2}
			C^{-1}\left|D^2v\right|^2\left(\mu^2+\left|Dv\right|^2\right)^\frac{p-2}{2}\le\left|D\left[V_p\left(Dv\right)\right]\right|^2\le C\left|D^2v\right|^2\left(\mu^2+\left|Dv\right|^2\right)^\frac{p-2}{2}.
	\end{equation}
\end{remark}

In what follows, the following result can be useful.

\begin{lemma}\label{differentiabilitylemma}
	Let $\Omega\subset\R^n$ be a bounded open set, $1<p<2$, and $v\in W^{1, p}_ {\loc}\left(\Omega, \R^N\right)$. Then the following implication holds
	\begin{equation*}\label{differentiabilityimplication}
		V_p\left(Dv\right)\in W^{1,2}_{\loc}\left(\Omega\right) \implies v\in W^{2,p}_{\loc}\left(\Omega\right),
	\end{equation*}
	
	and the following estimate
	
	\begin{equation}\label{differentiabilityestimate}
		\int_{B_{r}}\left|D^2v(x)\right|^pdx
		\le c\cdot \left[1+\int_{B_{R}}\left|DV_p\left(Dv(x)\right)\right|^2dx+c\int_{B_R}\left|Dv(x)\right|^pdx\right]
	\end{equation}
	
	holds for any ball $B_R\Subset\Omega$ and $0<r<R$.
\end{lemma}

\begin{proof}
	We will prove the existence of the second-order weak derivatives of $v$ and the fact that they are in $L^p_{\loc}\left(\Omega\right)$, by means of the difference quotients method.\\
	Let us consider a ball $B_R\Subset\Omega$ and $0<\frac{R}{2}<r<R$.\\
	
	For $|h|<\frac{R-r}{2}$, we have $0<\frac{R}{2}<r<\rho_1:=r+|h|<R-|h|=:\rho_2<R$, and by \eqref{lemma6GPestimate1}, we get, for any $s=1, ..., n.$
	\begin{eqnarray*}
		\int_{B_r}\left|\tau_{s, h}Dv(x)\right|^pdx&=& \int_{B_r}|\tau_{s, h}Dv(x)|^p\left(\mu^2+\left|Dv\left(x+he_s\right)\right|^2+\left|Dv\left(x\right)\right|^2\right)^\frac{p(p-2)}{4}\cr\cr
		&&\cdot\left(\mu^2+\left|Dv\left(x+he_s\right)\right|^2+\left|Dv\left(x\right)\right|^2\right)^\frac{p(2-p)}{4}dx.
	\end{eqnarray*}
	
	By H\"{o}lder's Inequality with exponents $\left(\frac{2}{p}, \frac{2}{2-p}\right)$ and the use of \eqref{lemma6GPestimate1}, we get
	
\begin{eqnarray*}
	\int_{B_r}\left|\tau_{s, h}Dv(x)\right|^pdx	&\le&\left(\int_{B_r}\left|\tau_{s, h}V_p\left(Dv(x)\right)\right|^2dx\right)^\frac{p}{2}\cr\cr
	&&\cdot\left(\int_{B_{r}}\left(\mu^2+\left|Dv\left(x+he_s\right)\right|^2+\left|Dv\left(x\right)\right|^2\right)^\frac{p}{2}dx\right)^\frac{2-p}{2},
\end{eqnarray*}
	
	and since $V_p\left(Dv\right)\in W^{1,2}_{\loc}\left(\Omega\right)$, by Lemma \ref{le1} and Young's Inequality, we have
	
\begin{eqnarray*}\label{lemmaestimate1}
	\int_{B_r}\left|\tau_{s, h}Dv(x)\right|^pdx&\le& c\left[|h|^2\int_{B_R}\left|DV_p\left(Dv(x)\right)\right|^2dx\right]^\frac{p}{2}\cr\cr
	&&\cdot\left[\int_{B_r}\left(\mu^2+\left|Dv\left(x+he_s\right)\right|^2+\left|Dv(x)\right|^2\right)^\frac{p}{2}dx\right]^\frac{2-p}{p}\cr\cr
	&\le&c|h|^p\left[1+\int_{B_{R}}\left|DV_p\left(Dv(x)\right)\right|^2dx+\int_{B_R}\left|Dv(x)\right|^pdx\right].
\end{eqnarray*}
	
	Since $v\in W^{1,p}_{\loc}\left(\Omega\right)$ and $V_p\left(Dv\right)\in W^{1,2}_{\loc}\left(\Omega\right)$, then, by Lemma \ref{Giusti8.2}, we get $v\in W^{2,p}_{\loc}\left(\Omega\right)$, and by Young's inequality, we have the estimate
	
	\begin{equation*}\label{differentiabilityestimate1}
		\int_{B_{r}}\left|D^2v(x)\right|^pdx
		\le c \left[1+\int_{B_{R}}\left|DV_p\left(Dv(x)\right)\right|^2dx+c\int_{B_R}\left|Dv(x)\right|^pdx\right],
	\end{equation*}
	
	that is \eqref{differentiabilityestimate}.
\end{proof}

\begin{remark}\label{rmk2}
	If $\Omega\subset\R^n$ is a bounded open set and $1<p<2$, then one may use Remark \ref{rmk1} and Lemma \ref{differentiabilitylemma} to show that, if $v\in W^{1, p}_{\loc}\left(\Omega\right)$ and $V_p\left(Dv\right)\in W^{1,2}_{\loc}\left(\Omega\right)$, then $v\in W^{2, p}_{\loc}\left(\Omega\right)$ and \eqref{lemma6GPestimate2} holds true.
\end{remark}
\begin{remark}\label{rmk3}	
	If $\Omega\subset\R^n$ is a bounded open set and $p\in\left(1, \infty\right)$, for any $v\in W^{1, p}_{\loc}\left(\Omega\right)$ such that $V_p\left(Dv\right)\in W^{1,2}_{\loc}\left(\Omega\right),$ if $m>1$ and $v\in L^{2m}_{\loc}\left(\Omega\right)$, then, thanks to \eqref{2.1GP}, $Dv\in L^{\frac{m\left(p+2\right)}{m+1}}_{\loc}\left(\Omega\right)$ and if $v\in L^{\infty}_{\loc}\left(\Omega\right)$, thanks to \eqref{2.2GP}, we get $Dv\in L^{p+2}_{\loc}\left(\Omega\right).$
\end{remark}
%
%

\subsection{Proof of Theorem \ref{thm1GG}}\label{Thm1GGPf}

\begin{proof}[Proof of the Theorem \ref{thm1GG}]
	{\bf Step 1: the a priori estimate.}\\
	Let us observe that, if $V_p\left(D\psi\right)\in W^{1,2}_\loc\left(\Omega\right)$ then, by virtue of Remark \ref{rmk3} and estimate \eqref{2.2GP}, we get $D\psi\in L^{p+2}_\loc\left(\Omega\right)$.\\
	Suppose that $u\in\mathcal{K}_\psi\left(\Omega\right)$ is a solution to the obstacle problem \eqref{functionalobstacle} such that
	$$
	V_p\left(Du\right)\in
	W^{1,2}_\loc\left(\Omega\right).
	$$
	Our first step is to prove the following a priori estimate
	
	\begin{eqnarray}\label{aprioriestimateCLAIMGG}
		&&\int_{B_{\frac{R}{2}}}\left|DV_p\left(Du(x)\right)\right|^2dx\cr\cr
		&\le&\frac{c\left(\left\Arrowvert \psi\right\Arrowvert_{L^\infty\left(B_{8R}\right)}^2+\left\Arrowvert u\right\Arrowvert_{L^{p^*}\left(B_{8R}\right)}^2\right)^{\sigma_1}}{R^2}\cr\cr
		&&\cdot\left[\int_{B_{4R}}\left(\mu^2+\left|Du(x)\right|^2\right)^\frac{p}{2}dx+\int_{B_{4R}}g^{p+2}(x)dx\right.\cr\cr
		&&\left.+\int_{B_{4R}} \left|DV_p\left(D\psi(x)\right)\right|^2dx+\int_{B_{4R}}\left(\mu^2+\left|D\psi(x)\right|^2\right)^\frac{p}{2}dx\right]^{\sigma_2},
	\end{eqnarray}
	for any ball $B_{8R}\Subset\Omega$.\\
	
	By estimate \eqref{est boundedness}, since $\psi\in L^\infty_{\loc}\left(\Omega\right)$, we have $u\in L^\infty_{\loc}\left(\Omega\right)$.\\
Recalling Remarks \ref{rmk2} and \ref{rmk3}, and Lemma \ref{lemma5GP}, thanks to the a priori assumption $V_p\left(Du\right)\in W^{1,2}_\loc\left(\Omega\right)$, we have $u\in W^{2, p}_\loc\left(\Omega\right)$ and $Du\in L^{p+2}_\loc\left(\Omega\right)$.\\
In order to apply Theorem \ref{Fuchs}, let us recall that $u\in W^{1,p}_{\loc}\left(\Omega\right)$ is a solution to the equation \eqref{diveq} if and only if, for any $\varphi\in W^{1,p}_0\left(\Omega\right)$,

\begin{equation}\label{diveqtestGG}
	\int_{\Omega}\left<A\left(x, Du(x)\right), D\varphi(x)\right>dx=-\int_{\Omega}\div A\left(x, D\psi(x)\right)\chi_{\set{u=\psi}}(x)\varphi(x)dx.
\end{equation}

Let us fix a ball $B_{8R}\Subset \Omega$ and arbitrary radii $\frac{R}{2}\le r<\tilde{s}<t<\tilde{t}<\lambda r,$ with $1<\lambda<2$. Let us consider a cut-off function $\eta\in C^\infty_0\left(B_t\right)$ such that $\eta\equiv 1$ on $B_{\tilde{s}}$ and $\left|D\eta\right|\le \frac{c}{t-{\tilde{s}}}$. From now on, with no loss of generality, we suppose $R<\frac{1}{4}$.\\

For any $s=1,\dots, n$ and $h\in\R$ with $\left|h\right|$ is sufficiently small, let us consider the test function
$$
\varphi=\tau_{s, -h}\left(\eta^2\cdot\tau_{s, h}u\right).
$$
	
	For this choice of $\varphi$, using Proposition \ref{findiffpr}, the left-hand side of \eqref{diveqtestGG} can be written as follows:
	
	\begin{eqnarray}\label{diveqtestLGG}
		&&\int_{\Omega}\left<A\left(x, Du(x)\right), D\left(\tau_{-h}\left(\eta^2(x)\tau_hu(x)\right)\right)\right>dx\cr\cr
		&=&\int_{\Omega}\left<\tau_{h}A\left(x, Du(x)\right), D\left(\eta^2(x)\tau_hu(x)\right)\right>dx\cr\cr
		&=&\int_{\Omega}\left<A\left(x+h, Du(x+h)\right)-A\left(x, Du(x)\right), D\left(\eta^2(x)\tau_hu(x)\right)\right>dx\cr\cr
		&=&\int_{\Omega}\left<A\left(x+h, Du(x+h)\right)-A\left(x, Du(x)\right), \eta^2(x)\tau_hDu(x)\right>dx\cr\cr
		&&+\int_{\Omega}\left<A\left(x+h, Du(x+h)\right)-A\left(x, Du(x)\right), 2\eta(x)D\eta(x)\tau_hu(x)\right>dx\cr\cr
		&=&\int_{\Omega}\left<A\left(x, Du(x+h)\right)-A\left(x, Du(x)\right), \eta^2(x)\tau_hDu(x)\right>dx\cr\cr
		&&+\int_{\Omega}\left<A\left(x+h, Du(x+h)\right)-A\left(x, Du(x+h)\right), \eta^2(x)\tau_hDu(x)\right>dx\cr\cr
		&&+\int_{\Omega}\left<A\left(x+h, Du(x+h)\right)-A\left(x, Du(x)\right), 2\eta(x)D\eta(x)\tau_hu(x)\right>dx\cr\cr
		&:=&I_0+I+II,
	\end{eqnarray}
	
	where, for the finite differences, we used the simplified notation
		$$\tau_{h}F\left(x\right)=F(x+h)-F(x),$$
		with $h\in\R^n$, in place of
		$$\tau_{s, h}F\left(x\right)=F\left(x+he_s\right)-F(x),$$
		with $h\in\R$ and, in the following, we will specify the direction only if it will be necessary.\\
	Since the right-hand side of \eqref{diveqtestGG} is not zero only where $u=\psi$, using the test function given above, it becomes
	
	\begin{equation}\label{diveqtestRGG*}
		-\int_{\Omega}\div A\left(x, D\psi(x)\right)\chi_{\set{u=\psi}}(x)\tau_{-h}\left(\eta^2(x)\tau_{h}\psi(x)\right)dx,
	\end{equation}
	
	and since the map $x\mapsto A(x, \xi)$ belongs to $W^{1,p+2}_{\loc}\left(\Omega\right)$ for any $\xi\in\R^n$, the map $\xi\mapsto A(x, \xi)$ belongs to $C^1(\R^n)$ for a.e. $x\in\Omega$ and $V_p\left(D\psi\right)\in W^{1,2}_{\loc}\left(\Omega\right)$, we can write \eqref{diveqtestRGG*} as follows
	
		\begin{eqnarray*}\label{diveqtestR}
	&&-\int_{\Omega}\Big\lbrace \Big[A_x\left(x, D\psi(x)\right)+A_\xi\left(x, D\psi(x)\right)D^2\psi(x)\Big]\chi_{\set{u=\psi}}(x)\cr\cr
	&&\cdot\tau_{-h}\left(\eta^2(x)\tau_h\psi(x)\right)\Big\rbrace dx\cr\cr
	&=&-\int_{\Omega}\Bigg\lbrace \Big[A_x\left(x, D\psi(x)\right)+A_\xi\left(x, D\psi(x)\right)D^2\psi(x)\Big]\chi_{\set{u=\psi}}(x)\cr\cr
	&&\cdot\tau_{-h}\left(\eta^2(x)\cdot h\int_{0}^{1}D\psi(x+\sigma h)d\sigma\right)\Bigg\rbrace dx\cr\cr
	&=&-\int_{\Omega}\Bigg\lbrace \Big[A_x\left(x, D\psi(x)\right)+A_\xi\left(x, D\psi(x)\right)D^2\psi(x)\Big]\chi_{\set{u=\psi}}(x)\cr\cr
	&&\cdot \left|h\right|^2\int_{0}^{1}\left[\eta^2(x-\theta h)\int_{0}^{1}D^2\psi(x+\sigma h-\theta h)d\sigma\right.\cr\cr
	&&\left.+2\eta(x-\theta h)D\eta(x-\theta h)\int_{0}^{1}D\psi(x+\sigma h-\theta h)d\sigma\right]d\theta\Bigg\rbrace dx\cr\cr
	&=&-\int_{\Omega}\Bigg\lbrace \Big[A_x\left(x, D\psi(x)\right)+A_\xi\left(x, D\psi(x)\right)D^2\psi(x)\Big]\chi_{\set{u=\psi}}(x)\cr\cr
	&&\cdot\int_{0}^{1}\int_{0}^{1}\left|h\right|^2\Big[\eta^2(x-\theta h)D^2\psi(x+\sigma h-\theta h)\cr\cr
	&&+2\eta(x-\theta h)D\eta(x-\theta h)D\psi(x+\sigma h-\theta h)\Big]d\sigma d\theta\Bigg\rbrace dx.
\end{eqnarray*}

Therefore, the right-hand side of \eqref{diveqtestGG} is given by the following expression

\begin{eqnarray}\label{rhs}
	&&-\left|h\right|^2\int_{\Omega}A_x\left(x, D\psi(x)\right)\chi_{\set{u=\psi}}(x)\int_{0}^{1}\int_{0}^{1}\eta^2(x-\theta h)D^2\psi(x+\sigma h-\theta h)d\sigma d\theta dx\cr\cr
	&&-2\left|h\right|^2\int_{\Omega}A_x\left(x, D\psi(x)\right)\chi_{\set{u=\psi}}(x)\int_{0}^{1}\int_{0}^{1}\eta(x-\theta h)D\eta(x-\theta h)D\psi(x+\sigma h-\theta h)d\sigma d\theta dx\cr\cr
	&&-\left|h\right|^2\int_{\Omega}A_\xi\left(x, D\psi(x)\right)D^2\psi(x)\chi_{\set{u=\psi}}(x)\int_{0}^{1}\int_{0}^{1}\eta^2(x-\theta h)D^2\psi(x+\sigma h-\theta h)d\sigma d\theta dx\cr\cr
	&&-2\left|h\right|^2\int_{\Omega}A_\xi\left(x, D\psi(x)\right)D^2\psi(x)\chi_{\set{u=\psi}}(x)\cr\cr
	&&\cdot\int_{0}^{1}\int_{0}^{1}\eta(x-\theta h)D\eta(x-\theta h)D\psi(x+\sigma h-\theta h)d\sigma d\theta dx\cr\cr
	&=:&-III-IV-V-VI.
\end{eqnarray}

Inserting \eqref{diveqtestLGG} and \eqref{rhs} in \eqref{diveqtestGG} we get
	
	\begin{equation}\label{starteqGG}
		I_0=-I-II-III-IV-V-VI,
	\end{equation}
	and so

\begin{equation}\label{startGG}
	I_0\le\left|I\right|+\left|II\right|+\left|III\right|+\left|IV\right|+\left|V\right|+\left|VI\right|.
\end{equation}

By assumption \eqref{obstacleA1}, we have

\begin{equation}\label{I_0GG}
	I_0\ge\nu\int_{\Omega}\eta^2(x)\left(\mu^2+\left|Du(x)\right|^2+\left|Du(x+h)\right|^2\right)^{\frac{p-2}{2}}\left|\tau_{h}Du(x)\right|^2dx.
\end{equation}

Let us consider the term $I$. By assumption \eqref{x-dependenceH}, and using Young's inequality with exponents $\left(2, 2\right)$, H\"{o}lder's inequality with exponents $\left(\frac{p+2}{p}, \frac{p+2}{2}\right)$, and the properties of $\eta$, we get

\begin{eqnarray}\label{IGG}
	\left|I\right|&\le&\int_{\Omega}|h|\left(g(x+h)+g(x)\right)\left(\mu^2+\left|Du(x)\right|^2+\left|Du(x+h)\right|^2\right)^{\frac{p-1}{2}}\eta^2(x)\left|\tau_{h}Du(x)\right|dx\cr\cr
	&\le&\varepsilon\int_{\Omega}\eta^2(x)\left(\mu^2+\left|Du(x)\right|^2+\left|Du(x+h)\right|^2\right)^{\frac{p-2}{2}}\left|\tau_{h}Du(x)\right|^2dx\cr\cr
	&&+c_\varepsilon|h|^2\int_{\Omega}\eta^2(x)\left(\mu^2+\left|Du(x)\right|^2+\left|Du(x+h)\right|^2\right)^{\frac{p}{2}}\left(g(x+h)+g(x)\right)^2dx\cr\cr
	&\le&\varepsilon\int_{\Omega}\eta^2(x)\left(\mu^2+\left|Du(x)\right|^2+\left|Du(x+h)\right|^2\right)^{\frac{p-2}{2}}\left|\tau_{h}Du(x)\right|^2dx\cr\cr
	&&+c_\varepsilon|h|^2\left(\int_{B_t}\left(\mu^2+\left|Du(x)\right|^2+\left|Du(x+h)\right|^2\right)^{\frac{p+2}{2}}dx\right)^\frac{p}{p+2}\cr\cr
	&&\cdot\left(\int_{B_{\lambda r}}g^{p+2}(x)dx\right)^\frac{2}{p+2}\cr\cr
	&\le&\varepsilon\int_{\Omega}\eta^2(x)\left(\mu^2+\left|Du(x)\right|^2+\left|Du(x+h)\right|^2\right)^{\frac{p-2}{2}}\left|\tau_{h}Du(x)\right|^2dx\cr\cr
	&&+c_\varepsilon|h|^2\left(\int_{B_{\tilde{t}}}\left(\mu^2+\left|Du(x)\right|^2\right)^{\frac{p+2}{2}}dx\right)^\frac{p}{p+2}\cdot\left(\int_{B_{\lambda r}}g^{p+2}(x)dx\right)^\frac{2}{p+2},
\end{eqnarray}

where we also used Lemma \ref{le1}.\\
Let us consider the term $II$. If we denote again finite differences with respect to a precise direction $s=1, \dots, n$, with an integration by parts, we have

\begin{eqnarray*}\label{-IIGG}
	&&-II=-2h\int_{\Omega}\left<\int_{0}^{1}\frac{d}{dx_s}A\left(x+\theta h e_s, Du(x+\theta h e_s)\right)d\theta, \eta(x)D\eta(x)\tau_{s, h}u(x)\right>dx\cr\cr
	&=&2h\int_{\Omega}\left<\int_{0}^{1}\left(A\left(x+\theta h e_s, Du(x+\theta h e_s)\right)\right)d\theta, \frac{d}{dx_s}\left(\eta(x)D\eta(x)\tau_{s,h}u(x)\right)\right>dx,
\end{eqnarray*}

where, for $s=1, \dots, n$, $e_s$ is the unit vector in the $x_s$ direction, and now $h\in\R$.\\
So we can estimate $II$ as follows

\begin{eqnarray*}\label{II*GG}
	\left|II\right|&\le&2|h|\int_{\Omega}\int_{0}^{1}\left|A\left(x+\theta h e_s, Du\left(x+\theta h e_s\right)\right)\right|\cr\cr
	&&\cdot\left(\left|D\eta(x)\right|^2\left|\tau_{s,h}u(x)\right|+\eta(x)\left|D^2\eta(x)\right|\left|\tau_{s,h}u(x)\right|\right)d\theta dx\cr\cr
	&&+2|h|\int_{\Omega}\int_{0}^{1}\left|A\left(x+\theta h e_s, Du\left(x+\theta h e_s\right)\right)\right|\cr\cr
	&&\cdot\left(\eta(x)\left|D\eta(x)\right|\left|\tau_{s,h}Du(x)\right|\right)d\theta dx\cr\cr
	&\le&2|h|\int_{\Omega}\int_{0}^{1}\left|A\left(x+\theta h e_s, Du\left(x+\theta h e_s\right)\right)\right|\cr\cr
	&&\cdot\left(\left|D\eta(x)\right|^2+\eta(x)\left|D^2\eta(x)\right|\right)d\theta \left|\tau_{s,h}u(x)\right|dx\cr\cr
	&&+2|h|\int_{\Omega}\int_{0}^{1}\left|A\left(x+\theta h e_s, Du\left(x+h_s\theta e_s\right)\right)\right|\cr\cr
	&&\cdot\eta(x)\left|D\eta(x)\right|\left|\tau_{s,h}Du(x)\right|d\theta dx.
\end{eqnarray*}

Now, recalling the properties of $\eta$, assumption \eqref{obstacleA3}, and using H\"{o}lder's inequality with exponents $\left(p, \frac{p}{p-1}\right)$, Lemma \ref{le1} and Young's inequality with exponents $\left(2, 2\right)$, we get

\begin{eqnarray*}\label{II**GG}
	\left|II\right|&\le&2c|h|\int_{B_t}\int_{0}^{1}\left(\mu^2+\left|Du(x)\right|^2+\left|Du\left(x+\theta h e_s\right)\right|^2\right)^\frac{p-1}{2}\cr\cr
	&&\cdot\left(\left|D\eta(x)\right|^2+\eta(x)\left|D^2\eta(x)\right|\right)d\theta\left|\tau_{s, h}u(x)\right|dx\cr\cr
	&&+2c|h|\int_{\Omega}\int_{0}^{1}\left(\mu^2+\left|Du(x)\right|^2+\left|Du\left(x+\theta h e_s\right)\right|^2\right)^\frac{p-1}{2}\cr\cr
	&&\cdot\eta(x)\left|D\eta(x)\right|\left|\tau_{s, h}Du(x)\right|d\theta dx\cr\cr
	&=&2c|h|\int_{0}^{1}\int_{B_t}\left(\mu^2+\left|Du(x)\right|^2+\left|Du\left(x+\theta h e_s\right)\right|^2\right)^\frac{p-1}{2}\cr\cr
	&&\cdot\left(\left|D\eta(x)\right|^2+\eta(x)\left|D^2\eta(x)\right|\right)\left|\tau_{s,h}u(x)\right|dxd\theta\cr\cr
	&&+2c|h|\int_{0}^{1}\int_{\Omega}\left(\mu^2+\left|Du(x)\right|^2+\left|Du\left(x+\theta h e_s\right)\right|^2\right)^\frac{p-1}{2}\cr\cr
	&&\cdot\eta(x)\left|D\eta(x)\right|\left|\tau_{s,h}Du(x)\right|dxd\theta\cr\cr
	&\le&\frac{c|h|}{\left(t-\tilde{s}\right)^2}\int_0^1\left(\int_{B_t}\left(\mu^2+\left|Du(x)\right|^2+\left|Du\left(x+\theta h e_s\right)\right|^2\right)^\frac{p}{2}dx\right)^\frac{p-1}{p}d\theta\cr\cr
	&&\cdot\left(\int_{B_t}\left|\tau_{s, h}u(x)\right|^pdx\right)^\frac{1}{p}\cr\cr
	&&+\varepsilon\int_{\Omega}\eta^2(x)\left|\tau_{s, h}Du(x)\right|^2\left(\mu^2+\left|Du(x)\right|^2+\left|Du(x+he_s)\right|^2\right)^\frac{p-2}{2}dx\cr\cr
	&&+\frac{c_\varepsilon\left|h\right|^2}{\left(t-\tilde{s}\right)^2}\int_{0}^{1}\int_{B_t}\left(\mu^2+\left|Du(x)\right|^2+\left|Du\left(x+\theta he_s	\right)\right|^2\right)^{p-1}\cr\cr
	&&\cdot\left(\mu^2+\left|Du(x)\right|^2+\left|Du\left(x+\theta he_s\right)\right|^2\right)^\frac{2-p}{2}dxd\theta.
\end{eqnarray*}

Now, if we use again the simplified notation for finite diferences, with $h\in\R^n$ in place of $he_s$ where $h\in\R$, by Lemma \ref{le1}, we get

\begin{eqnarray}\label{IIGG}
	\left|II\right|&\le&\frac{c|h|^2}{\left(t-\tilde{s}\right)^2}\int_0^1\left(\int_{B_t}\left(\mu^2+\left|Du(x)\right|^2+\left|Du(x+\theta h)\right|^2\right)^\frac{p}{2}dx\right)^\frac{p-1}{p}d\theta\cr\cr
	&&\cdot\left(\int_{B_{\tilde{t}}}\left|Du(x)\right|^pdx\right)^\frac{1}{p}\cr\cr
	&&+\varepsilon\int_{\Omega}\eta^2(x)\left|\tau_{h}Du(x)\right|^2\left(\mu^2+\left|Du(x)\right|^2+\left|Du(x+h)\right|^2\right)^\frac{p-2}{2}dx\cr\cr
	&&+\frac{c_\varepsilon\left|h\right|^2}{\left(t-\tilde{s}\right)^2}\int_{0}^{1}\left[\int_{B_{\lambda r}}\left(\mu^2+\left|Du(x)\right|^2+\left|Du(x+\theta h)\right|^2\right)^\frac{p-1}{2}\right.\cr\cr
	&&\left.\cdot\left(\mu^2+\left|Du(x)\right|^2+\left|Du(x+\theta h)\right|^2\right)^\frac{2-p}{4}dx\right]^2d\theta.
\end{eqnarray}

Let us consider, now, the term $III$.
By \eqref{x-dependence} and the properties of $\eta$, we get

\begin{eqnarray*}\label{III*GG}
	|III|&\le&|h|^2\int_{0}^{1}\int_{0}^{1}\int_{B_{\lambda r}}g(x)\left(\mu^2+\left|D\psi(x)\right|^2+\left|D\psi(x+h)\right|^2\right)^\frac{p-1}{2}\cr\cr
	&&\cdot\left|D^2\psi(x+\sigma h-\theta h)\right|dxd\sigma d\theta.
\end{eqnarray*}

Using Young's inequality with exponents $\left(2, 2\right)$, we get

\begin{eqnarray}\label{III**GG}
	|III|&\le&c|h|^2\int_{0}^{1}\int_{0}^{1}\Bigg[\int_{B_{\lambda r}} g^2(x)\left(\mu^2+\left|D\psi(x)\right|^2+\left|D\psi(x+h)\right|^2\right)^\frac{p}{2}dx\cr\cr
	&&+\int_{B_{\lambda r}}\left(\mu^2+\left|D\psi(x)\right|^2+\left|D\psi(x+h)\right|^2\right)^\frac{p-2}{2}\cr\cr
	&&\cdot\left|D^2\psi(x+\sigma h-\theta h)\right|^2dx\Bigg]d\sigma d\theta.
\end{eqnarray}

Again by Young's inequality with exponents $\left(\frac{p+2}{2}, \frac{p+2}{p}\right)$ in the first integral of \eqref{III**GG}, we get

\begin{eqnarray}\label{IIIGG}
	\left|III\right|&\le&c|h|^2\left[\int_{B_{\lambda r}}g^{p+2}(x)dx+\int_{B_{\lambda r}}\left(\mu^2+\left|D\psi(x)\right|^2+\left|D\psi(x+h)\right|^2\right)^\frac{p+2}{2}dx\right]\cr\cr
	&&+c|h|^2\int_0^1\int_0^1\int_{B_{\lambda r}}\left(\mu^2+\left|D\psi(x)\right|^2+\left|D\psi(x+h)\right|^2\right)^\frac{p-2}{2}\cr\cr
	&&\cdot\left|D^2\psi(x+\sigma h-\theta h)\right|^2dxd\sigma d\theta.
\end{eqnarray}

By \eqref{x-dependence}, we can estimate the term $IV$, thus getting

\begin{eqnarray}\label{IVGG}
	\left|IV\right|&\le&2|h|^2\int_{B_{\lambda r}} g(x)\left(\mu^2+\left|D\psi(x)\right|^2+\left|D\psi(x+h)\right|^2\right)^\frac{p-1}{2}\cr\cr
	&&\cdot\int_{0}^{1}\int_{0}^{1}\left|D\psi(x+\sigma h-\theta h)\right|\left|D\eta(x-\theta h)\right|d\sigma d\theta dx.
\end{eqnarray}

Let us consider, now, the term $V$. By assumption \eqref{A2bis}, we get

\begin{eqnarray}\label{VGG}
	\left|V\right|&\le&|h|^2\int_{B_{\lambda r}} \left(\mu^2+\left|D\psi(x)\right|^2\right)^\frac{p-2}{2}\left|D^2\psi(x)\right|\cr\cr
	&&\cdot\int_{0}^{1}\int_{0}^{1}\left|D^2\psi(x+\sigma h-\theta h)\right|d\sigma d\theta dx.
\end{eqnarray}

Recalling \eqref{A2bis} again, we have

\begin{eqnarray}\label{VIGG}
	\left|VI\right|&\le&2|h|^2\int_{B_{\lambda r}}\left(\mu^2+\left|D\psi(x)\right|^2\right)^\frac{p-2}{2}\left|D^2\psi(x)\right|\cr\cr
	&&\cdot\int_{0}^{1}\int_{0}^{1}\left|D\eta(x-\theta h)\right|\left|D\psi(x+\sigma h-\theta h)\right|d\sigma d\theta dx.
\end{eqnarray}

Now, inserting \eqref{I_0GG}, \eqref{IGG}, \eqref{IIGG}, \eqref{IIIGG}, \eqref{IVGG}, \eqref{VGG} and \eqref{VIGG} in \eqref{startGG}, recalling the properties of $\eta$ and choosing a sufficiently small value of $\varepsilon$, we get

\begin{eqnarray}\label{full1GG}
	&&\int_{\Omega}\eta^2(x)\left(\mu^2+\left|Du(x)\right|^2+\left|Du(x+h)\right|^2\right)^{\frac{p-2}{2}}\left|\tau_{h}Du(x)\right|^2dx\cr\cr
	&\le&c|h|^2\left(\int_{B_{\tilde{t}}}\left(\mu^2+\left|Du(x)\right|^2+\left|Du(x+h)\right|^2\right)^{\frac{p+2}{2}}dx\right)^\frac{p}{p+2}\cdot\left(\int_{B_{\lambda r}}g^{p+2}(x)dx\right)^\frac{2}{p+2}\cr\cr
	&&+\frac{c|h|^2}{\left(t-\tilde{s}\right)^2}\int_0^1\left(\int_{B_t}\left(\mu^2+\left|Du(x)\right|^2+\left|Du(x+\theta h)\right|^2\right)^\frac{p}{2}dx\right)^\frac{p-1}{p}d\theta\cr\cr
	&&\cdot\left(\int_{B_{\tilde{t}}}\left|Du(x)\right|^pdx\right)^\frac{1}{p}\cr\cr
	&&+\frac{c\left|h\right|^2}{\left(t-\tilde{s}\right)^2}\int_{0}^{1}\int_{B_t}\left(\mu^2+\left|Du(x)\right|^2+\left|Du(x+\theta h)\right|^2\right)^{p-1}\cr\cr
	&&\cdot\left(\mu^2+\left|Du(x)\right|^2+\left|Du(x+\theta h)\right|^2\right)^\frac{2-p}{2}dxd\theta\cr\cr
	&&+c|h|^2\left[\int_{B_{\lambda r}}g^{p+2}(x)dx+\int_{B_{\lambda r}}\left(\mu^2+\left|D\psi(x)\right|^2+\left|D\psi(x+h)\right|^2\right)^\frac{p+2}{2}dx\right]\cr\cr
	&&+c|h|^2\int_0^1\int_0^1\int_{B_{\lambda r}}\left(\mu^2+\left|D\psi(x)\right|^2+\left|D\psi(x+h)\right|^2\right)^\frac{p-2}{2}\cr\cr
	&&\cdot\left|D^2\psi(x+\sigma h-\theta h)\right|^2dxd\sigma d\theta\cr\cr
	&&+2|h|^2\int_{0}^{1}\int_{0}^{1}\int_{B_{\lambda r}} g(x)\left(\mu^2+\left|D\psi(x)\right|^2+\left|D\psi(x+h)\right|^2\right)^\frac{p-1}{2}\cr\cr
	&&\cdot\left|D\psi(x+\sigma h-\theta h)\right|\left|D\eta(x-\theta h)\right|dxd\sigma d\theta\cr\cr
	&&+|h|^2\int_{0}^{1}\int_{0}^{1}\int_{B_{\lambda r}} \left(\mu^2+\left|D\psi(x)\right|^2\right)^\frac{p-2}{2}\left|D^2\psi(x)\right|\cr\cr
	&&\cdot\left|D^2\psi(x+\sigma h-\theta h)\right|dxd\sigma d\theta \cr\cr
	&&+2|h|^2\int_{0}^{1}\int_{0}^{1}\int_{B_{\lambda r}}\left(\mu^2+\left|D\psi(x)\right|^2\right)^\frac{p-2}{2}\left|D^2\psi(x)\right|\cr\cr
	&&\cdot\left|D\eta(x-\theta h)\right|\left|D\psi(x+\sigma h-\theta h)\right|dxd\sigma d\theta.
\end{eqnarray}

By Lemma \ref{lemma6GP} and the properties of $\eta$, the left-hand side of \eqref{full1GG} can be bounded from below as follows

\begin{equation}\label{VpbelowGG}
	\int_{\Omega}\eta^2(x)\left(\mu^2+\left|Du(x)\right|^2+\left|Du(x+h)\right|^2\right)^{\frac{p-2}{2}}\left|\tau_{h}Du(x)\right|^2dx\ge\int_{\Omega}\eta^2(x)\left|\tau_{h}V_p\left(Du(x)\right)\right|^2dx.
\end{equation}

So, by \eqref{VpbelowGG} and \eqref{full1GG}, recalling the properties of $\eta$ and using Lemma \ref{le1}, we get

\begin{eqnarray*}\label{full2GG}
	&&\int_{\Omega}\eta^2(x)\left|\tau_{h}V_p\left(Du(x)\right)\right|^2dx\cr\cr
	&\le& c|h|^2\left(\int_{B_{\lambda r}}\left(\mu^2+\left|Du(x)\right|^2\right)^{\frac{p+2}{2}}dx\right)^\frac{p}{p+2}\cdot\left(\int_{B_{2r}}g^{p+2}(x)dx\right)^\frac{2}{p+2}\cr\cr
	&&+\frac{c|h|^2}{\left(t-\tilde{s}\right)^2}\int_{B_{2r}}\left(\mu^2+\left|Du(x)\right|^2\right)^\frac{p}{2}dx\cr\cr
	&&+c|h|^2\left[\int_{B_{2r}}g^{p+2}(x)dx+\int_{B_{2r}}\left(\mu^2+\left|D\psi(x)\right|^2\right)^\frac{p+2}{2}dx\right]\cr\cr
	&&+\frac{c|h|^2}{t-\tilde{s}}\int_{B_{2r}} g(x)\left(\mu^2+\left|D\psi(x)\right|^2\right)^\frac{p}{2}dx\cr\cr
	&&+c|h|^2\int_{B_{2r}} \left(\mu^2+\left|D\psi(x)\right|^2\right)^\frac{p-2}{2}\left|D^2\psi(x)\right|^2dx\cr\cr
	&&+\frac{c|h|^2}{t-\tilde{s}}\int_{B_{2r}}\left(\mu^2+\left|D\psi(x)\right|^2\right)^\frac{p-1}{2}\left|D^2\psi(x)\right|dx.
\end{eqnarray*}

Now we apply H\"{o}lder's inequality with three exponents $\left(p+2, p+2, \frac{p+2}{p}\right)$ to the integral of the fifth line, Young's inequality with exponents $\left(2, 2\right)$ to the last integral, thus getting

\begin{eqnarray*}\label{full4GG}
	&&\int_{\Omega}\eta^2(x)\left|\tau_{h}V_p\left(Du(x)\right)\right|^2dx\cr\cr
	&\le& c|h|^2\left(\int_{B_{\lambda r}}\left(\mu^2+\left|Du(x)\right|^2\right)^{\frac{p+2}{2}}dx\right)^\frac{p}{p+2}\cdot\left(\int_{B_{2r}}g^{p+2}(x)dx\right)^\frac{2}{p+2}\cr\cr
	&&+\frac{c|h|^2}{\left(t-\tilde{s}\right)^2}\int_{B_{2r}}\left(\mu^2+\left|Du(x)\right|^2\right)^\frac{p}{2}dx\cr\cr
	&&+c|h|^2\left[\int_{B_{2r}}g^{p+2}(x)dx+\int_{B_{2r}}\left(\mu^2+\left|D\psi(x)\right|^2\right)^\frac{p+2}{2}dx\right]\cr\cr
	&&+\frac{c|h|^2}{t-\tilde{s}}\left(\int_{B_{2r}} g^{p+2}(x)dx\right)^\frac{1}{p+2}\cdot\left(\int_{B_{2r}}\left(\mu^2+\left|D\psi(x)\right|^2\right)^\frac{p+2}{2}dx\right)^\frac{p}{p+2}\cr\cr
	&&+\frac{c|h|^2}{t-\tilde{s}}\left[\int_{B_{2r}} \left|DV_p\left(D\psi(x)\right)\right|^2dx+\int_{B_{2r}}\left(\mu^2+\left|D\psi(x)\right|^2\right)^\frac{p}{2}dx\right]
\end{eqnarray*}

for a constant $c=c(n, p, \nu, L, \ell)$, where we also used \eqref{lemma6GPestimate2}. By Young's inequality with exponents $\left(\frac{p+2}{2}, \frac{p+2}{p}\right)$, for some $\varepsilon>0$, we get

\begin{eqnarray*}\label{conclusion1.2GG}
	&&\int_\Omega\eta^2(x)\left|\tau_{h}V_p\left(Du(x)\right)\right|^2dx\cr\cr
	&\le& c|h|^2\left[\varepsilon\int_{B_{\lambda r}}\left(\mu^2+\left|Du(x)\right|^2\right)^{\frac{p+2}{2}}dx+\frac{c}{t-\tilde{s}}\int_{B_{2r}}\left(\mu^2+\left|D\psi(x)\right|^2\right)^{\frac{p+2}{2}}dx\right]\cr\cr
	&&+\frac{c_\varepsilon|h|^2}{t-\tilde{s}}\left[\int_{B_{2r}}g^{p+2}(x)dx+\left(\int_{B_{2r}} g^{p+2}(x)dx\right)^\frac{1}{2}\right]\cr\cr
	&&+\frac{c|h|^2}{t-\tilde{s}}\left[\int_{B_{2r}} \left|DV_p\left(D\psi(x)\right)\right|^2dx+\int_{B_{2r}}\left(\mu^2+\left|D\psi(x)\right|^2\right)^\frac{p}{2}dx\right]\cr\cr
	&&+\frac{c|h|^2}{\left(t-\tilde{s}\right)^2}\left(\int_{B_{2r}}\left(\mu^2+\left|Du(x)\right|^2\right)^\frac{p}{2}dx\right),
\end{eqnarray*}

and since $\eta\equiv1$ on $B_{\tilde{s}}$, we get

\begin{eqnarray*}\label{conclusion1.2GG'}
	&&\int_{B_{\tilde{s}}}\left|\tau_{h}V_p\left(Du(x)\right)\right|^2dx\cr\cr
	&\le& c|h|^2\left[\varepsilon\int_{B_{\lambda r}}\left(\mu^2+\left|Du(x)\right|^2\right)^{\frac{p+2}{2}}dx+\frac{c}{t-\tilde{s}}\int_{B_{2r}}\left(\mu^2+\left|D\psi(x)\right|^2\right)^{\frac{p+2}{2}}dx\right]\cr\cr
	&&+\frac{c_\varepsilon|h|^2}{t-\tilde{s}}\left[\int_{B_{2r}}g^{p+2}(x)dx+\left(\int_{B_{2r}} g^{p+2}(x)dx\right)^\frac{1}{2}\right]\cr\cr
	&&+\frac{c|h|^2}{t-\tilde{s}}\left[\int_{B_{2r}} \left|DV_p\left(D\psi(x)\right)\right|^2dx+\int_{B_{2r}}\left(\mu^2+\left|D\psi(x)\right|^2\right)^\frac{p}{2}dx\right]\cr\cr
	&&+\frac{c|h|^2}{\left(t-\tilde{s}\right)^2}\left(\int_{B_{2r}}\left(\mu^2+\left|Du(x)\right|^2\right)^\frac{p}{2}dx\right).
\end{eqnarray*}

Thanks to Lemma \ref{Giusti8.2}, deduce

\begin{eqnarray}\label{conclusion1.2GG''}
	&&\int_{B_{\tilde{s}}}\left|DV_p\left(Du(x)\right)\right|^2dx\cr\cr
	&\le& c\left[\varepsilon\int_{B_{\lambda r}}\left(\mu^2+\left|Du(x)\right|^2\right)^{\frac{p+2}{2}}dx+\frac{c}{t-\tilde{s}}\int_{B_{2r}}\left(\mu^2+\left|D\psi(x)\right|^2\right)^{\frac{p+2}{2}}dx\right]\cr\cr
	&&+\frac{c_\varepsilon|h|^2}{t-\tilde{s}}\left[\int_{B_{2r}}g^{p+2}(x)dx+\left(\int_{B_{2r}} g^{p+2}(x)dx\right)^\frac{1}{2}\right]\cr\cr
	&&+\frac{c}{t-\tilde{s}}\left[\int_{B_{2r}} \left|DV_p\left(D\psi(x)\right)\right|^2dx+\int_{B_{2r}}\left(\mu^2+\left|D\psi(x)\right|^2\right)^\frac{p}{2}dx\right]\cr\cr
	&&+\frac{c}{\left(t-\tilde{s}\right)^2}\left(\int_{B_{2r}}\left(\mu^2+\left|Du(x)\right|^2\right)^\frac{p}{2}dx\right).
\end{eqnarray}

Now, since $\mu\in[0,1]$, we have

\begin{eqnarray}\label{Dup+2}
	\left(\mu^2+\left|Du\right|^2\right)^{\frac{p+2}{2}}&=&\mu^2\left(\mu^2+\left|Du\right|^2\right)^{\frac{p}{2}}+\left(\mu^2+\left|Du\right|^2\right)^{\frac{p}{2}}\left|Du\right|^2\cr\cr
	&\le&\left(\mu^2+\left|Du\right|^2\right)^{\frac{p}{2}}+\left(\mu^2+\left|Du\right|^2\right)^{\frac{p}{2}}\left|Du\right|^2
\end{eqnarray}	

and, similarly,

\begin{eqnarray}\label{Dpsip+2}
	\left(\mu^2+\left|D\psi\right|^2\right)^{\frac{p+2}{2}}
	&\le&\left(\mu^2+\left|D\psi\right|^2\right)^{\frac{p}{2}}+\left(\mu^2+\left|D\psi\right|^2\right)^{\frac{p}{2}}\left|D\psi\right|^2.
\end{eqnarray}

Since $\tilde{t}<\lambda r<\lambda\tilde{s}<\lambda t<\lambda^2r<4R<1$, if we use \eqref{2.2GP} with $\phi\in C^\infty_{0}\left(B_{\lambda t}\right)$ such that $0\le\phi\le1$, $\phi\equiv1$ on $B_{\lambda \tilde{s}}$ and $\left|D\phi\right|\le\frac{c}{\lambda\left(t-\tilde{s}\right)}$, recalling \eqref{lemma6GPestimate2}, thanks to \eqref{Dup+2} we get

\begin{eqnarray}\label{2.2GPDu}
	&&\int_{B_{\tilde{t}}}\left(\mu^2+\left|Du(x)\right|^2\right)^\frac{p+2}{2}dx\cr\cr
	&\le&c\left\Arrowvert u\right\Arrowvert_{L^\infty\left(B_{4R}\right)}^2\int_{B_{\lambda t}}\left|DV_p\left(Du(x)\right)\right|^2dx\cr\cr
	&&+\frac{c\left\Arrowvert u\right\Arrowvert_{L^\infty\left(B_{4R}\right)}^2}{\lambda^2\left(t-\tilde{s}\right)^2}\int_{B_{4R}}\left(\mu^2+\left|Du(x)\right|^2\right)^\frac{p}{2}dx.
\end{eqnarray}

Arguing in the same way, using \eqref{2.2GP} with $\phi\in C^\infty_{0}\left(B_{2 t}\right)$ such that $0\le\phi\le1$, $\phi\equiv1$ on $B_{2\tilde{s}}$ and $\left|D\phi\right|\le\frac{c}{2\left(t-\tilde{s}\right)}$, thanks to \eqref{Dpsip+2}, since $r<\tilde{s}<t<R<\frac{1}{4}$, we get

\begin{eqnarray}\label{2.2GPDpsi}
	&&\int_{B_{2r}}\left(\mu^2+\left|D\psi(x)\right|^2\right)^\frac{p+2}{2}dx\cr\cr
	&\le&c\left\Arrowvert \psi\right\Arrowvert_{L^\infty\left(B_{4R}\right)}^2\int_{B_{4R}}\left|DV_p\left(D\psi(x)\right)\right|^2dx\cr\cr
	&&+\frac{c\left\Arrowvert \psi\right\Arrowvert_{L^\infty\left(B_{4R}\right)}^2}{4\left(t-\tilde{s}\right)^2}\int_{B_{4R}}\left(\mu^2+\left|D\psi(x)\right|^2\right)^\frac{p}{2}dx.
\end{eqnarray}

Therefore, inserting \eqref{2.2GPDu} and \eqref{2.2GPDpsi} into \eqref{conclusion1.2GG''}, since $1<\lambda<2$ and $t-\tilde{s}<1$, we get

\begin{eqnarray*}
	&&\int_{B_{\tilde{s}}}\left|DV_p\left(Du(x)\right)\right|^2dx\cr\cr
	&\le& c\left\Arrowvert u\right\Arrowvert_{L^\infty\left(B_{4R}\right)}^2\cdot\varepsilon\int_{B_{\lambda t}}\left|DV_p\left(Du(x)\right)\right|^2dx\cr\cr
	&&+\frac{c_\varepsilon\left\Arrowvert u\right\Arrowvert_{L^\infty\left(B_{4R}\right)}^2}{\left(t-\tilde{s}\right)^2}\left(\int_{B_{4R}}\left(\mu^2+\left|Du(x)\right|^2\right)^\frac{p}{2}dx\right)\cr\cr
	&&+\frac{c_\varepsilon\left\Arrowvert \psi\right\Arrowvert_{L^\infty\left(B_{4R}\right)}^2}{\left(t-\tilde{s}\right)^2}\left[\int_{B_{4R}}g^{p+2}(x)dx+\left(\int_{B_{4R}} g^{p+2}(x)dx\right)^\frac{1}{2}\right.\cr\cr
	&&\left.+\int_{B_{4R}} \left|DV_p\left(D\psi(x)\right)\right|^2dx+\int_{B_{4R}}\left(\mu^2+\left|D\psi(x)\right|^2\right)^\frac{p}{2}dx\right],
\end{eqnarray*}
and recalling \eqref{est boundedness}, we have

\begin{eqnarray}\label{pre-iterazione1}
	&&\int_{B_{\tilde{s}}}\left|DV_p\left(Du(x)\right)\right|^2dx\cr\cr
	&\le&\varepsilon\cdot c\left(\left\Arrowvert \psi\right\Arrowvert_{L^\infty\left(B_{8R}\right)}^2+\left\Arrowvert u\right\Arrowvert_{L^{p^*}\left(B_{8R}\right)}^2\right)^{\sigma_1}\int_{B_{\lambda t}}\left|DV_p\left(Du(x)\right)\right|^2dx\cr\cr
	&&+\frac{c_\varepsilon\left(\left\Arrowvert \psi\right\Arrowvert_{L^\infty\left(B_{8R}\right)}^2+\left\Arrowvert u\right\Arrowvert_{L^{p^*}\left(B_{8R}\right)}^2\right)^{\sigma_1}}{\left(t-\tilde{s}\right)^2}\cr\cr
	&&\cdot\left[\int_{B_{4R}}\left(\mu^2+\left|Du(x)\right|^2\right)^\frac{p}{2}dx+\int_{B_{4R}}g^{p+2}(x)dx\right.\cr\cr
	&&\left.+\int_{B_{4R}} \left|DV_p\left(D\psi(x)\right)\right|^2dx+\int_{B_{4R}}\left(\mu^2+\left|D\psi(x)\right|^2\right)^\frac{p}{2}dx\right]^{\sigma_2},
\end{eqnarray}
where $\sigma_1$ and $\sigma_2$ depend on $n$ and $p$.
Now, if we choose $\varepsilon>0$ such that
$$
\varepsilon\cdot c\left(\left\Arrowvert \psi\right\Arrowvert_{L^\infty\left(B_{8R}\right)}^2+\left\Arrowvert u\right\Arrowvert_{L^{p^*}\left(B_{8R}\right)}^2\right)^{\sigma_1}=\frac{1}{2},
$$
\eqref{pre-iterazione1} becomes
\begin{eqnarray}\label{pre-iterazione2}
	&&\int_{B_{\tilde{s}}}\left|DV_p\left(Du(x)\right)\right|^2dx\cr\cr
	&\le&\frac{1}{2}\int_{B_{\lambda t}}\left|DV_p\left(Du(x)\right)\right|^2dx\cr\cr
	&&+\frac{c\left(\left\Arrowvert \psi\right\Arrowvert_{L^\infty\left(B_{8R}\right)}^2+\left\Arrowvert u\right\Arrowvert_{L^{p^*}\left(B_{8R}\right)}^2\right)^{\sigma_1}}{\left(t-\tilde{s}\right)^2}\cr\cr
	&&\cdot\left[\int_{B_{4R}}\left(\mu^2+\left|Du(x)\right|^2\right)^\frac{p}{2}dx+\int_{B_{4R}}g^{p+2}(x)dx\right.\cr\cr
	&&\left.+\int_{B_{4R}} \left|DV_p\left(D\psi(x)\right)\right|^2dx+\int_{B_{4R}}\left(\mu^2+\left|D\psi(x)\right|^2\right)^\frac{p}{2}dx\right]^{\sigma_2},
\end{eqnarray}
and since \eqref{pre-iterazione2} holds for any $\frac{R}{2}\le r<\tilde{s}<t<\lambda r<R$ and for any $\lambda\in(1, 2)$ and the constant $c$ is independent of the radii, we can pass to the limit as $\tilde{s}\to r$ and $t\to\lambda r$, thus getting
\begin{eqnarray*}
	&&\int_{B_{r}}\left|DV_p\left(Du(x)\right)\right|^2dx\cr\cr
	&\le&\frac{1}{2}\int_{B_{\lambda^2 r}}\left|DV_p\left(Du(x)\right)\right|^2dx\cr\cr
	&&+\frac{c\left(\left\Arrowvert \psi\right\Arrowvert_{L^\infty\left(B_{8R}\right)}^2+\left\Arrowvert u\right\Arrowvert_{L^{p^*}\left(B_{8R}\right)}^2\right)^{\sigma_1}}{r^2\left(\lambda-1\right)^2}\cr\cr
	&&\cdot\left[\int_{B_{4R}}\left(\mu^2+\left|Du(x)\right|^2\right)^\frac{p}{2}dx+\int_{B_{4R}}g^{p+2}(x)dx\right.\cr\cr
	&&\left.+\int_{B_{4R}} \left|DV_p\left(D\psi(x)\right)\right|^2dx+\int_{B_{4R}}\left(\mu^2+\left|D\psi(x)\right|^2\right)^\frac{p}{2}dx\right]^{\sigma_2},
\end{eqnarray*}

which also implies

\begin{eqnarray}\label{pre-iterazione3}
	&&\int_{B_{r}}\left|DV_p\left(Du(x)\right)\right|^2dx\cr\cr
	&\le&\frac{1}{2}\int_{B_{\lambda^2 r}}\left|DV_p\left(Du(x)\right)\right|^2dx\cr\cr
	&&+\frac{c\left(\left\Arrowvert \psi\right\Arrowvert_{L^\infty\left(B_{8R}\right)}^2+\left\Arrowvert u\right\Arrowvert_{L^{p^*}\left(B_{8R}\right)}^2\right)^{\sigma_1}}{r^2\left(\lambda^2-1\right)^2}\cr\cr
	&&\cdot\left[\int_{B_{4R}}\left(\mu^2+\left|Du(x)\right|^2\right)^\frac{p}{2}dx+\int_{B_{4R}}g^{p+2}(x)dx\right.\cr\cr
	&&\left.+\int_{B_{4R}} \left|DV_p\left(D\psi(x)\right)\right|^2dx+\int_{B_{4R}}\left(\mu^2+\left|D\psi(x)\right|^2\right)^\frac{p}{2}dx\right]^{\sigma_2}.
\end{eqnarray}	

Now, setting
$$
h(r)=\int_{B_{r}}\left|DV_p\left(Du(x)\right)\right|^2dx,
$$

\begin{eqnarray*}
	A&=&c\left(\left\Arrowvert \psi\right\Arrowvert_{L^\infty\left(B_{8R}\right)}^2+\left\Arrowvert u\right\Arrowvert_{L^{p^*}\left(B_{8R}\right)}^2\right)^{\sigma_1}\cr\cr
	&&\cdot\left[\int_{B_{4R}}\left(\mu^2+\left|Du(x)\right|^2\right)^\frac{p}{2}dx+\int_{B_{4R}}g^{p+2}(x)dx\right.\cr\cr
	&&\left.+\int_{B_{4R}} \left|DV_p\left(D\psi(x)\right)\right|^2dx+\int_{B_{4R}}\left(\mu^2+\left|D\psi(x)\right|^2\right)^\frac{p}{2}dx\right]^{\sigma_2}
\end{eqnarray*}

and
\begin{eqnarray*}
	B&=&0,
\end{eqnarray*}
since \eqref{pre-iterazione3} holds for any $1<\lambda<2$, we can apply the Interation Lemma \ref{iteration} with
$$
\theta=\frac{1}{2}\qquad\mbox{ and }\qquad\gamma=2,
$$
thus getting
\begin{eqnarray*}\label{IterationGG}
	&&\int_{B_{\frac{R}{2}}}\left|DV_p\left(Du(x)\right)\right|^2dx\cr\cr
	&\le&\frac{c\left(\left\Arrowvert \psi\right\Arrowvert_{L^\infty\left(B_{8R}\right)}^2+\left\Arrowvert u\right\Arrowvert_{L^{p^*}\left(B_{8R}\right)}^2\right)^{\sigma_1}}{R^2}\cr\cr
	&&\cdot\left[\int_{B_{4R}}\left(\mu^2+\left|Du(x)\right|^2\right)^\frac{p}{2}dx+\int_{B_{4R}}g^{p+2}(x)dx\right.\cr\cr
	&&\left.+\int_{B_{4R}} \left|DV_p\left(D\psi(x)\right)\right|^2dx+\int_{B_{4R}}\left(\mu^2+\left|D\psi(x)\right|^2\right)^\frac{p}{2}dx\right]^{\sigma_2},
\end{eqnarray*}

that is \eqref{aprioriestimateCLAIMGG}, where $c>0$ depends on $n, p, \nu, L$ and $\ell$, and $\sigma_1, \sigma_2>0$ depend on $n$ and $p$.\\
\medskip

{\bf Step 2: the approximation.}\\
Fix an open set $\Omega'\Subset \Omega$, and for a smooth kernel
$\phi \in C^{\infty}_{0}\left(B_{1}(0)\right)$ with $\phi \geq 0$ and
$\int_{B_{1}(0)}\phi = 1$, and for any $\varepsilon\in\left(0, d\left(\Omega', \partial\Omega\right)\right)$, let us consider the corresponding family of mollifiers $ \Set{\phi_{\eps}}_{\eps}$, and set

$$g_\e=g\ast\phi_\e,$$

$$ \mathcal{K}_{\varepsilon,\psi}\left(\Omega\right)=\Set{w\in u+W^{1,
		p}_0\left(\Omega\right): w\ge\psi \mbox{ a.e. in }\Omega}$$
and
\begin{equation*}\label{A_epsilonGG}
	A_\varepsilon(x,\xi)=\int_{B_1}\phi(\omega)A\left(x+\varepsilon\omega,\xi\right)d\omega
\end{equation*}
on $\Omega'$, for each $\eps\in\left(0, d\left(\Omega',\partial\Omega\right)\right)$. The
assumptions \eqref{obstacleA3}--\eqref{obstacleA2} imply

\begin{equation}\label{A3'GG}
	\left|A_\e(x, \xi)\right|\le
	\ell\left(\mu^2+|\xi|^2\right)^\frac{p-1}{2},
\end{equation}

\begin{equation}\label{A1'GG}
	\langle A_\varepsilon(x,\xi)-A_\varepsilon(x,\eta),\xi-\eta\rangle \ge \nu|\eta-\xi|^2\left(\mu^2+|\xi|^2+|\eta|^2\right)^{\frac{p-2}{2}}.
\end{equation}

\begin{equation}\label{A2'GG}
	\left|A_\varepsilon(x,\xi)-A_\varepsilon(x,\eta)\right|\le L \left|\xi-\eta \right|\left(\mu^2+|\xi|^2+|\eta|^2\right)^{\frac{p-2}{2}}.
\end{equation}

By virtue of assumption \eqref{x-dependenceH} we also have

\begin{equation}\label{A4'GG}
	\left|A_\varepsilon(x,\xi)-A_\varepsilon(y,\xi)\right|\le \left(g_{\e}(x)+g_\e(y)\right)|x-y|\left(\mu^2+\left|\xi\right|^2\right)^{\frac{p-1}{2}}
\end{equation}
for almost every $x,y\in\Omega$ and for all $\xi,\eta \in
\R^{n}$. Let $u$ be a solution of the variational inequality
\eqref{variationalinequality} and let fix a ball $B_{\tilde{R}}\Subset
\Omega'$. Let us denote by $u_\eps \in u+W^{1,p}_0\left(B_{\tilde{R}}\right)$ the solution to the inequality

\begin{equation}\label{variationalinequality2GG}
	\int_{\Omega}\left<A_\e\left(x, Dw(x)\right),
	D\left(\varphi-w\right)(x)\right>dx\ge0\qquad\forall \varphi\in
	\mathcal{K}_{\e,\psi}\left(\Omega\right).
\end{equation}
Thanks to \cite[Theorem 1.1]{Gentile3} we have
$V_p\left(Du_\varepsilon\right)\in
W^{1,2}_\loc\left(B_{\tilde{R}}\right) $ and, since $A_\eps$ satisfies conditions \eqref{A3'GG}--\eqref{A4'GG}, for $\eps$ sufficiently small, we are legitimated to apply estimate \eqref{aprioriestimateCLAIMGG} thus getting

\begin{eqnarray}\label{aprioriestimateCLAIM_eps}
	&&\int_{B_{\frac{r}{2}}}\left|DV_p\left(Du_\varepsilon(x)\right)\right|^2dx\cr\cr
	&\le&\frac{c\left(\left\Arrowvert \psi\right\Arrowvert_{L^\infty\left(B_{8r}\right)}^2+\left\Arrowvert u_\varepsilon\right\Arrowvert_{L^{p^*}\left(B_{8r}\right)}^2\right)^{\sigma_1}}{r^2}\cr\cr
	&&\cdot\left[\int_{B_{4r}}\left(\mu^2+\left|Du_\varepsilon(x)\right|^2\right)^\frac{p}{2}dx+\int_{B_{4r}}g^{p+2}_\varepsilon(x)dx\right.\cr\cr
	&&\left.+\int_{B_{4r}} \left|DV_p\left(D\psi(x)\right)\right|^2dx+\int_{B_{4r}}\left(\mu^2+\left|D\psi(x)\right|^2\right)^\frac{p}{2}dx\right]^{\sigma_2},
\end{eqnarray}
for a constant $c=c(n, p, \nu, L, \ell)$, for any ball $B_{8r}\Subset B_{\tilde{R}}.$\\
Moreover, since and $g\in L^{p+2}_\loc\left(\Omega\right)$, we have

\begin{equation}\label{convkGG}
	g_\eps \to g\qquad \mbox{ strongly in }L^{p+2}\left(B_{\tilde{R}}\right),\mbox{ as }\varepsilon\to0
\end{equation}

and, up to a subsequence, almost everywhere in $B_{\tilde{R}}$.

Since by \eqref{A3'GG}, $\left|A_\varepsilon\left(x,Du\right)\right|\le \ell
\left(\mu^2+\left|Du\right|^2\right)^\frac{p-1}{2}$ and since $A_\varepsilon\left(x,Du\right)$ converges
almost everywhere to $A\left(x,Du\right)$, by the dominated convergence Theorem we have

\begin{equation}\label{convdfGG}
	A_\varepsilon\left(x,Du\right)\to A\left(x,Du\right)\qquad
	\mbox{ strongly in }L^{\frac{p}{p-1}}\left(B_{\tilde{R}}\right),\mbox{ as }\varepsilon\to0.
\end{equation}

Using the ellipticity condition \eqref{A1'GG}, we have

\begin{eqnarray}\label{convforteGG}
	&&\int_{B_{\tilde{R}}}\left(\mu^2+\left|Du(x)\right|^2+\left|Du_\eps(x)\right|^2\right)^{\frac{p-2}{2}}\left|\left(Du_\eps-Du\right)(x)\right|^2dx\cr\cr
	&\le& \int_{B_{\tilde{R}}}\langle A_\eps\left(x,Du_\eps(x)\right)-A_\e\left(x,Du(x)\right),
	\left(Du_\eps-Du\right)(x)\rangle dx\cr\cr &=&\int_{B_{\tilde{R}}}\langle
	A_\eps\left(x,Du_\eps(x)\right), \left(Du_\eps-Du\right)(x)\rangle dx\cr\cr
	&&-\int_{B_{\tilde{R}}}\langle
	A_\e\left(x,Du(x)\right), \left(Du_\eps-Du\right)(x)\rangle dx\cr\cr &=&\int_{B_{\tilde{R}}}\langle
	A_\eps\left(x,Du_\eps(x)\right), \left(Du_\eps-Du\right)(x)\rangle dx\cr\cr
	&&-\int_{B_{\tilde{R}}}
	\langle
	A\left(x,Du(x)\right), \left(Du_\eps-Du\right)(x)\rangle dx \cr\cr
	&&-\int_{B_{\tilde{R}}}\langle A_\eps\left(x,Du(x)\right)-A\left(x,Du(x)\right),
	\left(Du_\eps-Du\right)(x)\rangle dx
\end{eqnarray}

\medskip

\noindent Using $u$ and $u_\e$ as test functions in
\eqref{variationalinequality2GG} and \eqref{variationalinequality}
respectively, we have

\begin{equation}\label{a'}
	\int_{B_{\tilde{R}}}\langle A_\eps\left(x,Du_\eps(x)\right),
	\left(Du_\eps-Du\right)(x)\rangle dx\le 0
\end{equation}

and

\begin{equation}\label{b'}
	-\int_{B_{\tilde{R}}}\langle A\left(x,Du(x)\right), \left(Du_\eps-Du\right)(x)\rangle dx \le 0,
\end{equation}

therefore, thanks to \eqref{a'} and \eqref{b'}, \eqref{convforteGG} implies

\begin{eqnarray}\label{due}
	&&\int_{B_{\tilde{R}}}\left(\mu^2+\left|Du(x)\right|^2+\left|Du_\eps(x)\right|^2\right)^{\frac{p-2}{2}}\left|\left(Du_\eps-Du\right)(x)\right|^2dx\cr\cr
	&\le& -\int_{B_{\tilde{R}}}\langle A_\eps\left(x,Du(x)\right)-A	\left(x,Du(x)\right),
	\left(Du_\eps-Du\right)(x)\rangle\,dx\cr\cr
	&\le&\left(\int_{B_{\tilde{R}}}\left|\left(Du-Du_\eps\right)(x)\right|^pdx\right)^{\frac{1}{p}}\cr\cr
	&&\cdot\left(\int_{B_{\tilde{R}}}\left|
	A\left(x,Du(x)\right)-A_\eps\left(x,Du(x)\right)\right|^{\frac{p}{p-1}}dx\right)^{\frac{p-1}{p}}.
\end{eqnarray}
Now let us observe that, by \eqref{a'}, using H\"older's inequality with exponents $\left(p, \frac{p}{p-1}\right)$ and recalling \eqref{A3'GG}, we get
\begin{eqnarray}\label{boundueps1}
	\int_{B_{R}}\langle A_\eps\left(x,Du_\eps(x)\right), Du_\varepsilon(x)\rangle dx&\le&\int_{B_{\tilde{R}}}\langle A_\eps\left(x,Du_\varepsilon(x)\right), Du(x)\rangle dx\cr\cr
	&\le&\left(\int_{B_{\tilde{R}}}\left|A_\eps\left(x,Du_\varepsilon(x)\right)\right|^{\frac{p}{p-1}}dx\right)^\frac{p-1}{p}\cr\cr
	&&\cdot\left(\int_{B_{\tilde{R}}}\left|Du(x)\right|^pdx\right)^\frac{1}{p}\cr\cr
	&\le&\left(\int_{B_{\tilde{R}}}\left(\mu^2+\left|Du_\varepsilon(x)\right|^2\right)^\frac{p}{2}dx\right)^\frac{p-1}{p}\cr\cr
	&&\cdot\left(\int_{B_{\tilde{R}}}\left|Du(x)\right|^pdx\right)^\frac{1}{p}.
\end{eqnarray}
Moreover, using H\"older's inequality with exponents $\left(\frac{2}{p}, \frac{2}{2-p}\right)$ and thanks to the ellipticity condition \eqref{A1'GG}, we have
\begin{eqnarray}\label{boundueps2}
	\int_{B_{\tilde{R}}}\left|Du_\varepsilon(x)\right|^pdx&\le&	\int_{B_{\tilde{R}}}\left|Du_\varepsilon(x)\right|^p\left(\mu^2+\left|Du_\e(x)\right|^2\right)^{\frac{p\left(p-2\right)}{4}}\cdot\left(\mu^2+\left|Du_\e(x)\right|^2\right)^{\frac{p\left(2-p\right)}{4}}dx\cr\cr
	&\le&\left(\int_{B_{\tilde{R}}}\left|Du_\varepsilon(x)\right|^2\left(\mu^2+\left|Du_\e(x)\right|^2\right)^{\frac{p-2}{2}}dx\right)^\frac{p}{2}\cr\cr
	&&\cdot\left(\int_{B_{\tilde{R}}}\left(\mu^2+\left|Du_\e(x)\right|^2\right)^{\frac{p}{2}}dx\right)^\frac{2-p}{2}\cr\cr
	&\le&\left(\int_{B_{\tilde{R}}}\left<A_\e\left(x, Du_\e(x)\right)-A_\e\left(x, 0\right),Du_\e(x) \right>dx\right)^\frac{p}{2}\cr\cr
	&&\cdot\left(\int_{B_{\tilde{R}}}\left(\mu^2+\left|Du_\e(x)\right|^2\right)^{\frac{p}{2}}dx\right)^\frac{2-p}{2}.
\end{eqnarray}
We can notice that, since the ellipticity condition implies
$$
\left<A_\e\left(x, Du_\e(x)\right)-A_\e\left(x, 0\right),Du_\e(x) \right>\ge0,
$$
and so
$$
\left<A_\e\left(x, Du_\e(x)\right),Du_\e(x) \right>\ge\left<A_\e\left(x, 0\right),Du_\e(x) \right>,
$$
for a.e. $x\in B_{\tilde{R}}$.\\
Hence, if we denote
$$
E_1:=\Set{x\in B_{\tilde{R}}: \left<A_\e\left(x, Du_\e(x)\right),Du_\e(x) \right><0},$$
and
$$
E_2:=\Set{x\in B_{\tilde{R}}: \left<A_\e\left(x, Du_\e(x)\right),Du_\e(x) \right>\ge0},
$$
we have
$$
\left|\left<A_\e\left(x, Du_\e(x)\right),Du_\e(x) \right>\right|\le\left|\left<A_\e\left(x, 0\right),Du_\e(x) \right>\right|
$$
for a.e. $x\in E_1$, and
$$
\left<A_\e\left(x, Du_\e(x)\right)-A_\e\left(x, 0\right),Du_\e(x) \right>\le\left<A_\e\left(x, Du_\e(x)\right),Du_\e(x) \right>+\left|\left<A_\e\left(x, 0\right),Du_\e(x) \right>\right|
$$
for a.e. $x\in E_2$.
Therefore \eqref{boundueps2} implies
\begin{eqnarray}\label{boundueps3}
	\int_{B_{\tilde{R}}}\left|Du_\varepsilon(x)\right|^pdx
	&\le&\left(c\int_{E_1}\left|\left<A_\e\left(x, 0\right),Du_\e(x) \right>\right|dx\right.\cr\cr
	&&\left.+\int_{E_2}\left(\left<A_\e\left(x, Du_\e(x)\right),Du_\e(x) \right>+\left|\left<A_\e\left(x, 0\right),Du_\e(x) \right>\right|\right)dx\right)^\frac{p}{2}\cr\cr
	&&\cdot\left(c\int_{B_{\tilde{R}}}\left(\mu^2+\left|Du_\e(x)\right|^2\right)^{\frac{p}{2}}dx\right)^\frac{2-p}{2}\cr\cr
	&\le&\left(\int_{E_1}\mu^{p-1}\left|Du_\e(x) \right|dx\right.\cr\cr
	&&\left.+\int_{E_2}\left(\left<A_\e\left(x, Du_\e(x)\right),Du_\e(x) \right>+\mu^{p-1}\left|Du_\e(x) \right|\right)dx\right)^\frac{p}{2}\cr\cr
	&&\cdot\left(\int_{B_{\tilde{R}}}\left(\mu^2+\left|Du_\e(x)\right|^2\right)^{\frac{p}{2}}dx\right)^\frac{2-p}{2},
\end{eqnarray}
where, in the last line, we used \eqref{A3'GG}.\\
Using Young's inequality with exponents $\left(\frac{2}{p}, \frac{2}{2-p}\right)$, by \eqref{boundueps3} we deduce

\begin{eqnarray}\label{boundueps4}
	\int_{B_{\tilde{R}}}\left|Du_\varepsilon(x)\right|^pdx
	&\le&c_\sigma\int_{B_{\tilde{R}}}\mu^{p-1}\left|Du_\e(x) \right|dx+c_\sigma\int_{B_{\tilde{R}}}\left<A_\e\left(x, Du_\e(x)\right),Du_\e(x) \right>dx\cr\cr
	&&+\sigma\int_{B_{\tilde{R}}}\left(\mu^2+\left|Du_\e(x)\right|^2\right)^{\frac{p}{2}}dx\cr\cr
	&\le&c_\sigma\int_{B_{\tilde{R}}}\left<A_\e\left(x, Du_\e(x)\right),Du_\e(x) \right>dx\cr\cr
	&&+2\sigma\int_{B_{\tilde{R}}}\left|Du_\e(x)\right|^{p}dx+c_\sigma\left|B_{\tilde{R}}\right|,
\end{eqnarray}
where we also used Young's inequality with exponents $\left(p, \frac{p}{p-1}\right)$ and the fact that $\mu\in[0, 1]$.\\
Now, joining \eqref{boundueps1} with \eqref{boundueps4}, we get
\begin{eqnarray}\label{boundueps5}
	\int_{B_{\tilde{R}}}\left|Du_\varepsilon(x)\right|^pdx
	&\le&c_\sigma\left(\int_{B_{\tilde{R}}}\left(\mu^2+\left|Du_\varepsilon(x)\right|^2\right)^\frac{p}{2}dx\right)^\frac{p-1}{p}\cdot\left(\int_{B_{\tilde{R}}}\left|Du(x)\right|^pdx\right)^\frac{1}{p}\cr\cr
	&&+2\sigma\int_{B_{\tilde{R}}}\left|Du_\e(x)\right|^{p}dx+c_\sigma\left|B_{\tilde{R}}\right|\cr\cr
	&\le&c_\sigma\int_{B_{\tilde{R}}}\left|Du(x)\right|^pdx+3\sigma\int_{B_{\tilde{R}}}\left|Du_\e(x)\right|^{p}dx\cr\cr
	&&+c_\sigma\left|B_{\tilde{R}}\right|,
\end{eqnarray}
where we used Young's inequality with exponents $\left(p, \frac{p}{p-1}\right)$ and the fact that $\mu\in[0, 1]$ again.\\
Choosing $\sigma<\frac{1}{3}$, \eqref{boundueps5} implies
\begin{eqnarray}\label{boundueps}
	\int_{B_{\tilde{R}}}\left|Du_\varepsilon(x)\right|^pdx
	&\le&c\int_{B_{\tilde{R}}}\left|Du(x)\right|^pdx+c\left|B_{\tilde{R}}\right|.
\end{eqnarray}

Let us observe that, using H\"{o}lder's inequality with exponents $\left(\frac{2}{p}, \frac{2}{2-p}\right)$ recalling \eqref{boundueps}, we have

\begin{eqnarray}\label{uno}
	\int_{B_{\tilde{R}}}\left|\left(Du_\varepsilon-Du\right)(x)\right|^pdx&=&\int_{B_{\tilde{R}}}\left(\mu^2+\left|Du(x)\right|^2+\left|Du_\eps(x)\right|^2\right)^{\frac{p(p-2)}{4}}\left|\left(Du_\eps-Du\right)(x)\right|^p\cr\cr
	&&\cdot\left(\mu^2+\left|Du(x)\right|^2+\left|Du_\eps(x)\right|^2\right)^{\frac{p(2-p)}{4}}dx\cr\cr
	&\le&\left(\int_{B_{\tilde{R}}}\left(\mu^2+\left|Du(x)\right|^2+\left|Du_\eps(x)\right|^2\right)^{\frac{p-2}{2}}\left|\left(Du_\eps-Du\right)(x)\right|^2dx\right)^\frac{p}{2}\cr\cr
	&&\cdot\left(\int_{B_{\tilde{R}}}\left(\mu^2+\left|Du(x)\right|^2+\left|Du_\e(x)\right|^2\right)^\frac{p}{2}dx\right)^\frac{2-p}{2}\cr\cr
	&\le&c\left(\int_{B_{\tilde{R}}}\left(\mu^2+\left|Du(x)\right|^2+\left|Du_\eps(x)\right|^2\right)^{\frac{p-2}{2}}\left|\left(Du_\eps-Du\right)(x)\right|^2dx\right)^\frac{p}{2}\cr\cr
	&&\cdot\left(\int_{B_{\tilde{R}}}\left(\mu^2+\left|Du(x)\right|^2\right)^\frac{p}{2}dx+\left|B_{\tilde{R}}\right|\right)^\frac{2-p}{2}\cr\cr
	&\le&c\left[\left(\int_{B_{\tilde{R}}}\left|\left(Du-Du_\eps\right)(x)\right|^pdx\right)^{\frac{1}{p}}\right.\cr\cr
	&&\left.\cdot\left(\int_{B_{\tilde{R}}}\left|
	A\left(x,Du(x)\right)-A_\eps\left(x,Du(x)\right)\right|^{\frac{p}{p-1}}dx\right)^{\frac{p-1}{p}}\right]^\frac{p}{2}\cr\cr
	&&\cdot\left(\int_{B_{\tilde{R}}}\left(\mu^2+\left|Du(x)\right|^2\right)^\frac{p}{2}dx+\left|B_{\tilde{R}}\right|\right)^\frac{2-p}{2}\cr\cr&=&c\left(\int_{B_{\tilde{R}}}\left|\left(Du-Du_\eps\right)(x)\right|^pdx\right)^{\frac{1}{2}}\cr\cr
	&&\cdot\left(\int_{B_{\tilde{R}}}\left|
	A\left(x,Du(x)\right)-A_\eps\left(x,Du(x)\right)\right|^{\frac{p}{p-1}}dx\right)^{\frac{p-1}{2}}\cr\cr
	&&\cdot\left(\int_{B_{\tilde{R}}}\left(\mu^2+\left|Du(x)\right|^2\right)^\frac{p}{2}dx+\left|B_{\tilde{R}}\right|\right)^\frac{2-p}{2},
\end{eqnarray}
where we also used \eqref{due}.\\
By Young's inequality with exponents $\left(2, 2\right)$, \eqref{uno} implies
\begin{eqnarray*}
	\int_{B_{\tilde{R}}}\left|\left(Du_\varepsilon-Du\right)(x)\right|^pdx&\le&\sigma\int_{B_{\tilde{R}}}\left|\left(Du-Du_\eps\right)(x)\right|^pdx\cr\cr
	&&+c_\sigma\left(\int_{B_{\tilde{R}}}\left|
	A\left(x,Du(x)\right)-A_\eps\left(x,Du(x)\right)\right|^{\frac{p}{p-1}}dx\right)^{p-1}\cr\cr
	&&\cdot\left(\int_{B_{\tilde{R}}}\left(\mu^2+\left|Du(x)\right|^2\right)^\frac{p}{2}dx+\left|B_{\tilde{R}}\right|\right)^{2-p},
\end{eqnarray*}
for any $\sigma>0$, and if we choose $\sigma<\frac{1}{2}$, we have

\begin{eqnarray*}\label{tre}
	\int_{B_{\tilde{R}}}\left|\left(Du_\varepsilon-Du\right)(x)\right|^pdx
	&\le&c\left(\int_{B_{\tilde{R}}}\left|
	A\left(x,Du(x)\right)-A_\eps\left(x,Du(x)\right)\right|^{\frac{p}{p-1}}dx\right)^{p-1}\cr\cr
	&&\cdot\left(\int_{B_{\tilde{R}}}\left(\mu^2+\left|Du(x)\right|^2\right)^\frac{p}{2}dx+\left|B_{\tilde{R}}\right|\right)^{2-p}.
\end{eqnarray*}

Hence, by \eqref{convdfGG}, we deduce

$$
Du_\e\to Du\qquad\mbox{ strongly in }L^p_\loc\left(B_{\tilde{R}}\right),
$$

which implies

$$
u_\e\to u\qquad\mbox{ strongly in }L^{p^*}_\loc\left(B_{\tilde{R}}\right)
$$

and, up to a subsequence
$$
u_\e\to u\qquad \mbox{ almost everywhere in }B_{\tilde{R}},
$$

as $\varepsilon\to0$.\\
Moreover, by the continuity of the map $\xi\mapsto DV_p\left(\xi\right)$, we also get

$$
DV_p\left(Du_\varepsilon\right) \to DV_p\left(Du\right) \qquad\mbox{ a.e. in }B_{\tilde{R}},\mbox{ as }\varepsilon\to0.
$$

Therefore, recalling \eqref{convkGG}, if we pass to the limit in \eqref{aprioriestimateCLAIM_eps}, by Fatou's Lemma and a covering argument, we conclude, proving \eqref{estimate1GG}.
\end{proof}
As a consequence of Theorem \ref{thm1GG}, Lemma \ref{differentiabilitylemma} and Remark \ref{rmk3}, we get the following result.
\begin{corollary}\label{D2corollaryGG}
Let $u\in W^{1,p}_{\loc}\left(\Omega\right)$ be a solution to the obstacle problem \eqref{functionalobstacle} under assumptions \eqref{obstacleA3}--\eqref{obstacleA2} and let us assume that there exists a function $g\in L^{p+2}_\loc\left(\Omega\right)$ such that \eqref{x-dependenceH} and \eqref{x-dependence} hold, for $1<p<2$.\\
Then the following implication holds:

\begin{equation*}
	\psi\in L^\infty_{\loc}\left(\Omega\right) \mbox{ and } V_p\left(D\psi\right)\in W^{1,2}_{\loc}\left(\Omega\right)\implies u\in W^{2, p}_{\loc}\left(\Omega\right)\mbox{ and }Du\in L^{p+2}_{\loc}\left(\Omega\right).
\end{equation*}
\end{corollary}
\bigskip
\bigskip
Acknowledgements: The authors have partially been supported by the Gruppo Nazionale per l'Analisi Matematica, la Probabilit?a e le loro Applicazioni (GNAMPA) of the Istituto Nazionale diAlta Matematica (INdAM)

\printbibliography
\bigskip
\bigskip
\noindent {\bf A. Gentile}\\
Universit\`{a} degli Studi di Napoli ``Federico II'' \\
Dipartimento di Mat.~e Appl. ``R.~Caccioppoli'',\\
Via Cintia, 80126 Napoli, Italy

\noindent {\em E-mail address}: andrea.gentile@unina.it

\bigskip
\bigskip

\noindent {\bf R. Giova}\\
\noindent Universit\`{a} degli Studi di Napoli ``Parthenope'' \\
Palazzo Pacanowsky - Via Generale Parisi, 13 \\
80132 Napoli, Italy

\noindent {\em E-mail address}: raffaella.giova@uniparthenope.it
\end{document}